\documentclass[11.05pt,a4paper]{amsart}
\usepackage{rotating}
\usepackage[all]{xy}

\usepackage[utf8]{inputenc}
\usepackage{amscd,amssymb,amsopn,amsmath,amsthm,graphics,amsfonts,enumerate,verbatim,calc}
\headsep=1cm \numberwithin{equation}{section}
\newtheorem{theorem}{Theorem}[section]
\newtheorem{lemma}[theorem]{Lemma}
\newtheorem{proposition}[theorem]{Proposition}
\newtheorem{corollary}[theorem]{Corollary}
\newtheorem{problem}[theorem]{Problem}
\newtheorem{observation}[theorem]{Observation}
\theoremstyle{definition}

\theoremstyle{remark}
\newtheorem{remark}[theorem]{Remark}\newtheorem{fact}[theorem]{Fact}
\newtheorem{example}[theorem]{Example}
\newtheorem{question}[theorem]{Question}
\newtheorem{conjecture}[theorem]{Conjecture}

\newcommand{\Ass}{\operatorname{Ass}}

\newcommand{\grade}{\operatorname{grade}}

\newcommand{\Min}{\operatorname{Min}}

\newcommand{\Spec}{\operatorname{Spec}}

\newcommand{\Tr}{\operatorname{Tr}}

\newcommand{\Ht}{\operatorname{ht}}
\newcommand{\id}{\operatorname{id}}

\newcommand{\pd}{\operatorname{pd}}
\newcommand{\Gdim}{\operatorname{Gdim}}

\newcommand{\Ext}{\operatorname{Ext}}
\newcommand{\Supp}{\operatorname{Supp}}

\newcommand{\Tor}{\operatorname{Tor}}
\newcommand{\Hom}{\operatorname{Hom}}

\newcommand{\Ann}{\operatorname{Ann}}

\newcommand{\depth}{\operatorname{depth}}

\newcommand{\Coker}{\operatorname{Coker}}

\newcommand{\lo}{\longrightarrow}
\newcommand{\fm}{\frak{m}}
\newcommand{\fp}{\frak{p}}

\newcommand{\CMdef}{\operatorname{CMdef}}

\begin{document}

\author[M. Asgharzadeh  and E. Mahdavi]{Mohsen Asgharzadeh and Elham Mahdavi}

\title[Remarks on modules of finite projective dimension]{Remarks on modules of finite projective dimension}

\address{M. Asgharzadeh }
\email{mohsenasgharzadeh@gmail.com}

\address{E. Mahdavi}
\email{elham.mahdavi.gh@gmail.com}

\subjclass[2020]{13C14; 13D02; 13D09}

\keywords{Dimenstion and grade; Ext-modules; projective dimension;  Gorenstein rings; tensor product. }

\begin{abstract}
	We investigate homological and depth-theoretic properties of finitely generated modules of finite projective dimension over Noetherian local rings. A central theme is the study of criteria for freeness and reflexivity derived from the torsion-freeness or reflexivity of tensor products of the form \( M \otimes_R M \) and \( M \otimes_R M^* \). Under mild homological assumptions, we prove that such properties of these tensor products impose strong structural constraints on \( M \), often forcing it to be free. These results generalize classical theorems of Auslander beyond the regular case.
	
	The second part of the paper is devoted to the dimension and support of Ext-modules, particularly \( \operatorname{Ext}^i_R(M, R) \) for critical values of \( i \), when \( M \) has finite projective dimension. We establish sharp bounds on their Krull dimensions, analyze their behavior for prime and equidimensional modules, and relate these findings to the grade conjecture and other homological conjectures, i.e.,  whenever
	$\operatorname{grade}(M) = \operatorname{ht}(\operatorname{Ann}(M))$ where $\Gdim(M)<\infty$.
	 We consider the problem that asks whenever is \( \pd_R(M \otimes_R N) = 1 \)? Applications include new cases of a question of Jorgensen, which asks whether \( \operatorname{pd}(M) < i \) whenever \( \operatorname{Ext}^i_R(M, M) = 0 \) and \( M \) has finite projective dimension over a complete intersection ring.
	Finally, we examine the projective dimension of prime ideals in rings that fail chain conditions. 
\end{abstract}

\maketitle
 \tableofcontents

\section{Introduction}

Let \( (R, \mathfrak{m}) \) be a Noetherian local ring and let \( M \) be a finitely generated \( R \)-module. Modules of finite projective dimension occupy a central position in commutative algebra, serving as a bridge between homological properties and structural features such as depth, dimension, and freeness. The systematic study of such modules has led to profound connections with the regularity and singularities of the base ring, as exemplified by the Auslander--Buchsbaum--Serre theorem, which characterizes regular local rings as precisely those rings over which every finitely generated module has finite projective dimension.

One of the most influential directions in this area concerns detecting freeness or reflexivity of a module from properties of its tensor powers. In his seminal work, Auslander proved that over a regular local ring, if \( M \otimes_R M \) is torsion-free, then \( M \) is reflexive; and if \( M \otimes_R M^* \) (where \( M^* = \operatorname{Hom}_R(M, R) \)) is reflexive, then \( M \) is free. These results inspired a series of conjectures attributed to Vasconcelos, predicting that freeness of a module should often be detectable from homological or depth properties of its tensor products. Over the years, these questions have motivated extensive research, particularly in understanding the rigidity imposed by tensor-product behavior.

The first main objective of this paper is to extend such rigidity phenomena beyond the setting of regular rings to arbitrary local rings, under the assumption that \( M \) has finite projective dimension. We show that, under mild depth and dimension hypotheses, reflexivity or torsion-freeness of the tensor products \( M \otimes_R M \) or \( M \otimes_R M^* \) forces strong structural consequences for \( M \), often leading to freeness.

\begin{theorem}\label{1.1}
	\begin{enumerate}
		\item
		Let \( P \in \Spec(R) \) with \( \dim (R/P) \leq 2 \). If \( \pd(P) < \infty \) and \( P^{\otimes i} \) is torsion-free for some \( i \ge 2 \), then \( P \) is free.
		
		\item
		Let \( M \) be a module with \( \pd(M) < \infty \). If \( M \otimes_R M^* \) is reflexive, then \( M \) is free.
		
		\item
		Let \( M \) be a module with \( \pd(M) < 2 \). If \( M \otimes_R M \) is reflexive, then \( M \) is reflexive.
		
		\item
		Suppose \( R \) is regular and \( M \otimes_R M \) satisfies \( (S_r) \). Then \( M \) satisfies \( (S_{r+1}) \).
		
		\item
		Suppose \( R \) satisfies \( (S_2) \), \( \pd(M) < \infty \), and both \( M \otimes_R M \) and \( \Hom_R(M,M) \) are reflexive. Then \( M \) is free.
	\end{enumerate}
\end{theorem}

Concerning the first statement, Vasconcelos \cite{2010} conjectured that there exists an integer \( e \) such that if \( M^{\otimes e} \) is torsion-free over an integral domain, then \( M \) is free.

A second fundamental theme of this work is the detailed analysis of the dimension and support of the modules \( \Ext^i_R(M,R) \) when \( M \) has finite projective dimension. Understanding the size and geometry of these modules is closely tied to the long-standing Grade Conjecture, which predicts that
\[
\operatorname{grade}(M) + \dim(M) = \dim(R).
\]
We prove sharp bounds on the Krull dimension of \( \Ext^i_R(M,R) \), with particular emphasis on extremal homological degrees.

\begin{theorem}\label{3}Suppose  \( N \) is locally free and torsion-free.
	\begin{enumerate}
		\item
		Let \( R \) be a regular ring of dimension \(4\) containing \( \mathbb{Q} \), and let \( M \) satisfy \( (S_3) \).  Then
		\[
		\dim \big(\Ext^{1}_R(M,N)\big) \le 2.
		\]

		\item
		Let \( c := \operatorname{codim}(M) \) and assume \( \pd(M) < \infty \). Then
		\(
		\dim \big(\Ext^{c}_R(M,R)\big) \ge \dim(R) - c.
		\)
		
		\item
		Assume in addition to {\rm (2)} that \( M \) is equidimensional. Then
		\(\dim\!\big(\Ext^{c}_R(M,N)\big) = \dim(M).\) 
		
		\item
		Let \( R \) be \( d \)-dimensional, and let \( M \) be a 1-dimensional module with \( \pd(M) < \infty \). Then
		\(
		\dim(R/\mathfrak{p}) + \operatorname{ht}(\mathfrak{p}) = \dim(R)
		\)
		for every \( \mathfrak{p} \in \Supp\!\big(\Ext^{d-1}_R(M,R)\big) \).
	\end{enumerate}
\end{theorem}

These results unify and extend earlier partial theorems and provide new insight into the homological behavior of modules of finite projective dimension. As applications, we analyze the structure of \( \Ext \)-modules for prime and equidimensional modules, showing in several cases that \( \Ext^i_R(M,R) \) has finite length. This further highlights the close connection between depth conditions, tensor product rigidity, and homological inequalities.

In Section~4 we collect preliminary inequalities relating depth and dimension. In particular, we study the Cohen--Macaulay defect
\(
\CMdef(M) := \dim(M) - \depth(M).
\)
The main result of that section is the following.

\begin{theorem}\label{1.2}
	\begin{enumerate}
		\item
		If \( \pd_R(M) < \infty \), then \( \CMdef(M) \ge \CMdef(R) \).
		
		\item
		Let \( R \) be complete and \( P \in \Spec(R) \) with finite projective dimension. Then
		\(
		\CMdef(R_P) \le \CMdef(R/P).
		\)
		In particular, if \( R/P \) is Cohen--Macaulay, then so is \( R_P \).
		
		\item
		Let \( P \in \Spec(R) \) have finite projective dimension. If \( \operatorname{ht}(P) = \dim(R)-1 \), then \( R \) is Cohen--Macaulay.
		
		\item
		Let \( R \) be non-Cohen--Macaulay and let \( P \in \Spec(R) \) have finite projective dimension. If \( \operatorname{ht}(P) = \dim(R)-2 \), then \( R/P \) is perfect. Moreover, both \( R \) and \( R/P \) are almost Cohen--Macaulay.
	\end{enumerate}
\end{theorem}

We also recover a classical result from EGA~IV by Grothendieck and Dieudonné, showing via elementary methods that
\( ``
\dim(M_{\mathfrak{p}}) - \depth(M_{\mathfrak{p}})
\le
\dim(M) - \depth(M)''.
\)
In Section 5, we address a problem posed by Jorgensen \cite[Question 2.7]{Jo}:

\begin{question}
	Suppose \( \Ext^i_R(M,M) = 0 \), \( \pd(M) < \infty \), and \( R \) is a complete intersection. Must one have \( \pd(M) < i \)?
\end{question}

We provide several affirmative cases.

\begin{theorem}\label{1.4}
	\begin{enumerate}
		\item
		Let \( R \) be a complete intersection and let \( I \) be a reflexive ideal. If \( \Ext^2_R(I,I) = 0 \), then \( I \) is free.
		
		\item
		Let \( R \) be a \( d \)-dimensional Gorenstein ring with \( d>1 \), \( \pd(M) < \infty \), and \( \Ext^{d-1}_R(M,M)=0 \). If \( M \) is quasi-Buchsbaum, then \( \pd(M) \le d-2 \).
		
		\item
		Let \( R \) be Cohen--Macaulay, \( I \subset R \) a radical ideal with \( \pd(I) < \infty \), and assume \( R/I \) is Cohen--Macaulay. If \( \Ext^i_R(I,I) = 0 \), then \( \pd(I) < i \).\item 
		A finite-length counterexample to  Question  1.4 over a complete intersection would require dimension at least \(6\).
	\end{enumerate}
\end{theorem}

Our arguments combine classical homological tools with more recent techniques involving local cohomology and depth formulas.

We further investigate Gorenstein dimension (\( \operatorname{Gdim} \)), the Gorenstein analogue of projective dimension. Motivated by properties of grade and support, we formulate a conjecture stating that if \( \operatorname{Gdim}(M) < \infty \), then
\[
\operatorname{grade}(M) = \operatorname{ht}(\operatorname{Ann}(M)).
\]
We verify this conjecture for symbolic and ordinary powers of prime ideals and explore its connections with other homological conjectures. In the same spirit, we settle a question implicit in our previous work \cite{elham} concerning the integral closure \( \overline{R} \): if \( R \) is analytically unramified and \( \pd_R(\overline{R}) < \infty \), then \( R \) is normal; and if \( \operatorname{Gdim}_R(\overline{R}) < \infty \), then \( R \) is quasi-normal. As a byproduct, we close Section 6 with Remark~\ref{disc:hypersurface-question}.

Section 7 is devoted to another direction pursued here concerns the projective dimension of tensor products \( M \otimes_R N \). A question raised in the literature asks whether there exist nonfree torsion-free modules \( M \) and \( N \) over a local ring \( R \) of depth two such that \( \pd_R(M \otimes_R N) = 1 \) (see \cite[Question 3.3]{celp}). We provide positive answers under suitable structural assumptions on \( R \) and the modules. In particular, we construct explicit examples over certain hypersurfaces and study the role of symmetry in the presentation matrix of \( M \). Under additional bimodule structure, we also obtain a negative result.

Finally, we study the interplay between projective dimension and chain conditions in the prime spectrum. We show that in a non-catenary three-dimensional ring, any nonzero prime ideal appearing in a bad saturated chain must have infinite projective dimension.

\begin{theorem}\label{1.8}
	\begin{enumerate}
		\item
		Let \( R \) be a three-dimensional non-catenary ring. Then any nonzero prime ideal appearing in a bad saturated chain has infinite projective dimension.
		
		\item
		Let \( R \) be a domain of dimension \( d \) with a strict saturated chain of prime ideals
		\[
		0=\fp_0 \subsetneqq \cdots \subsetneqq \fp_{d-2}\subsetneqq\fm.
		\]
		Then \( \pd(\fp_{d-2}) = \infty \).
	\end{enumerate}
\end{theorem}

These results complement and extend earlier work of Seydi \cite[Proposition II.3]{S} on chain conditions in UFDs, illustrating how homological methods can detect subtle geometric pathologies in the spectrum.
Motivated by Theorem~\ref{1.2}(2), which relates the Cohen--Macaulayness of \( R/P \) to that of \( R_P \) when \( \pd(R/P) < \infty \), we construct a ring \( R \) such that for every \( \mathfrak{p} \in \Spec(R)\setminus\{0,\mathfrak{m}\} \) one has \( \pd(R/\mathfrak{p}) = \infty \) but \( \pd(R_{\mathfrak{p}}) < \infty \).

Throughout, all rings are commutative, Noetherian, and local with identity; all modules are finitely generated unless stated otherwise. We follow the notation and conventions of \cite{BH,mat} for homological and commutative algebra.

\medskip
\section{Homology of Tensor Products}

Throughout this paper, \( (R,\fm,k) \) denotes a commutative Noetherian local ring, and all modules are finitely generated unless otherwise specified. The notation \( \pd_R(-) \) (resp. \( \id_R(-) \)) denotes projective (resp. injective) dimension, and \( \Gdim_R(-) \) denotes Gorenstein dimension.

We begin with the following conjecture.

\begin{conjecture}[Vasconcelos {\cite{2010}}]
	There exists an integer \( e \) such that if \( M^{\otimes e} \) is torsion-free over an integral domain, then \( M \) is free.
\end{conjecture}

We will repeatedly use the following well-known fact.

\begin{fact}\label{Auslander}
	The ring \( R \) is an integral domain if it possesses a prime ideal of finite projective dimension.
\end{fact}

\begin{proposition}
	Let \( P \in \Spec(R) \) with \( \dim(R/P) \leq 2 \). If \( \pd(P) < \infty \) and \( P^{\otimes i} \) is torsion-free for some \( i \ge 2 \), then \( P \) is free.
\end{proposition}

\begin{proof}
	We may assume \( P \neq 0 \). By Fact~\ref{Auslander}, \( R \) is an integral domain. Since \( P \) is a submodule of a free module, it is torsion-free. Hence \( (P_{R_P})^{\otimes i} \) is torsion-free. Because \( \pd(P) < \infty \), we also have \( \pd_{R_P}(P_{R_P}) < \infty \). Thus \( R_P \) is regular. A result of Auslander \cite{au} on torsion in tensor powers now implies that \( P_{R_P} \) is free over \( R_P \). Hence \( P_{R_P} \) is principal, so \( \operatorname{ht}(P) = \operatorname{ht}(P_{R_P}) = 1 \). By \cite[Corollary 1.2]{ABU}, it follows that \( P \) is principal, and therefore free.
\end{proof}

\begin{corollary}
	If \( \dim R = 3 \), \( P \in \Spec(R) \) has finite projective dimension, and \( P \otimes_R P \) is torsion-free, then \( P \) is free.
\end{corollary}

\begin{proof}
	We may assume \( P \neq 0 \). Then \( \dim(R/P) \le \dim R - 1 = 2 \), so the conclusion follows from the proposition.
\end{proof}

We write \( (-)^* := \Hom_R(-,R) \). The following result is due to Auslander.

\begin{theorem}[Auslander {\cite[Proposition 3.2]{au}}]\label{2.4}
	Let \( R \) be normal. If \( M \otimes_R M^* \) is reflexive, then \( M \) is free.
\end{theorem}

We denote by \( H^i_{\fm}(-) \) the \( i \)-th local cohomology module with respect to \( \fm \).

\begin{proposition}\label{30}
	In Theorem~\ref{2.4}, the normality assumption can be replaced by the condition \( \pd(M) < \infty \).
\end{proposition}

\begin{proof}
	Let \( d = \dim R \). We argue by induction on \( d \). If \( \depth(R) = 0 \), then by Auslander–Buchsbaum, \( \depth(M) + \pd(M) = 0 \), so \( M \) is free. Hence assume \( \depth(R) \ge 1 \). Since \( M \) is torsion-free, \( \depth(M) \ge 1 \). The result is clear if \( \dim R = 1 \), so assume \( \dim R \ge 2 \) and that \( R \) satisfies \( (S_2) \).
	
	Because \( M \otimes_R M^* \) is reflexive, \( \depth(M \otimes_R M^*) \ge 2 \), and hence \( H^1_{\fm}(M \otimes_R M^*) = 0 \). For every \( \fp \ne \fm \),
	\[
	(M \otimes_R M^*)_{\fp} \cong M_{\fp} \otimes_{R_{\fp}} \Hom_{R_{\fp}}(M_{\fp}, R_{\fp})
	\]
	is reflexive. By the induction hypothesis, \( M_{\fp} \) is free. Thus \( M \) is locally free on the punctured spectrum. Now apply \cite[Lemma 8.1]{ACS}: if \( L \) is locally free, \( \depth(R) \ge 2 \), and \( H^1_{\fm}(L \otimes L^*) = 0 \), then \( L \) is free. Hence \( M \) is free.
\end{proof}

\begin{theorem}[Auslander {\cite[Theorem 3.6(c)]{au}}]\label{tor}
	Let \( R \) be regular and \( M \) a module such that \( M \otimes_R M \) is torsion-free. Then \( M \) is reflexive.
\end{theorem}

\begin{conjecture}
	In Theorem~\ref{tor}, the regularity assumption can be replaced by \( \pd(M) < \infty \).
\end{conjecture}

\begin{proposition}
	Let \( A \) be a finitely generated module with \( A \cong A^* \) and \( \pd(A) \le 1 \). If \( A \otimes_R A \) is torsion-free, then \( A \) is free.
\end{proposition}

\begin{proof}
	Let \( N := A \otimes_R A \). Consider the natural exact sequence
	\(
	0 \to N \to N^{**} \to L \to 0.
	\)
	We argue by induction on \( d = \dim R \). If \( L \ne 0 \), then by induction \( L \) has finite length. From the long exact sequence in local cohomology,
	\[
	0 = H^0_{\fm}(N) \to H^0_{\fm}(L) \to H^1_{\fm}(N),
	\]
	we obtain \( \depth(N) = 1 \).
	
	Since \( \pd(A) \le 1 \), we have \( \Tor_i^R(A,N) = 0 \) for \( i \ge 1 \). Auslander's depth formula gives
	\[
	\depth(N) = \depth(A \otimes_R N) + \pd(A) - q,
	\]
	where \( q = \sup\{ i \mid \Tor_i^R(A,N) \ne 0 \} \). If \( q \ge 1 \), then \( \depth(N) \le 0 \), a contradiction. Hence \( q=0 \), and applying the depth formula again yields
	\[
	1 = \depth(N) = \depth(A^{\otimes 3}) + 1 \ge 2,
	\]
	another contradiction. Thus \( L=0 \), so \( N \cong N^{**} \) is reflexive.
	
	Since \( A \cong A^* \), we have \( A \otimes_R A^* \cong A \otimes_R A \) reflexive. By Proposition~\ref{30}, \( A \) is free.
\end{proof}

\begin{fact}\label{ref}
	A module \( L \) is reflexive if and only if \( \depth(L_{\fp}) \ge 2 \) for all \( \fp \) with \( \depth(R_{\fp}) \ge 2 \), and \( \depth(L_{\fp}) \ge 1 \) whenever \( \depth(R_{\fp}) = 1 \).
\end{fact}

\begin{proposition}
	Let \( M \) be a module with \( \pd(M) < 2 \). If \( M \otimes_R M \) is reflexive, then \( M \) is reflexive.
\end{proposition}

\begin{proof}
	Let \( \fp \in \Spec(R) \). If \( \depth(R_{\fp}) \le 1 \), then \( M_{\fp} \) is torsion-free, hence
	\[
	0 < \depth(M_{\fp}) \le \depth(R_{\fp}) \le 1,
	\]
	so \( \depth(M_{\fp}) = \depth(R_{\fp}) \). Now assume \( \depth(R_{\fp}) \ge 2 \). Replacing \( R_{\fp} \) by \( R \), we may assume \( \fp = \fm \). Since \( M \) is torsionless, there exists an exact sequence
	\(
	0 \to M \to R^n \to L \to 0.
	\)
	Tensoring with \( M \) gives
	\[
	0 \to \Tor_1^R(M,L) \to M \otimes_R M \to M^{\oplus n} \to L \otimes_R M \to 0.
	\]
	Because \( M \otimes_R M \) is torsion-free, \( \Tor_1^R(M,L)=0 \). If \( M \) is not free, then \( \pd(M)=1 \), and by \cite{au} we obtain \( \depth(L) > 0 \). Hence \( \depth(M) \ge 1 + \depth(L) \ge 2 \). The conclusion now follows from Fact~\ref{ref}.
\end{proof}

By \( (S_r) \) we mean Serre's condition.

\begin{proposition}
	Suppose \( R \) is regular and \( M \otimes_R M \) satisfies \( (S_r) \). Then \( M \) satisfies \( (S_{r+1}) \).
\end{proposition}

\begin{proof}
	Since \( (S_r) \) implies torsion-freeness, \( M \otimes_R M \) is torsion-free. By \cite{au}, \( \Tor_i^R(M,M)=0 \) for all \( i \ge 1 \). The depth formula yields
	\[
	2\depth(M) = \depth(R) + \depth(M \otimes_R M).
	\]
	If \( r = \dim R \), then Auslander–Buchsbaum implies that \( M \) is free. Otherwise, assuming \( \depth(M) \le r \) leads to
	\[
	2\depth(M) = \dim R + r > 2r,
	\]
	a contradiction. Thus \( \depth(M) \ge r+1 \), and \( M \) satisfies \( (S_{r+1}) \).
\end{proof}

\begin{conjecture}
	If \( \pd(M) < \infty \) and \( M \otimes_R M \) is reflexive, then \( M \) is reflexive.
\end{conjecture}

\begin{corollary}
	If \( \pd(M) < \infty \), \( \depth(R)=3 \), and \( M \otimes_R M \) is reflexive, then \( M \) is free.
\end{corollary}

\begin{proof}
	Then \( M \) is reflexive, so \( \depth(M) \ge 2 \), and Auslander–Buchsbaum gives \( \pd(M) \le 1 \). If \( M \) were not free, tensoring a presentation \( 0 \to M \to F \to C \to 0 \) with \( M \) would yield \( \Tor_1^R(M,M)=0 \). The depth formula would then give
	\[
	2\depth(M) = \depth(R) + \depth(M \otimes_R M) = 3 + 2 = 5,
	\]
	a contradiction.
\end{proof}

Let \( \Tr(-) \) denote the Auslander transpose, see \cite{AB} for  its definition.

\begin{theorem}\label{44}
	Suppose \( R \) satisfies \( (S_2) \), \( M \) is finitely generated with \( \pd(M) < \infty \), and both \( M \otimes_R M \) and \( \Hom_R(M,M) \) are reflexive. Then \( M \) is free.
\end{theorem}

\begin{proof}
	We argue by induction on \( d = \dim R \). The result is known for \( d \le 2 \), so assume \( d>2 \). By induction, \( M \) is locally free on the punctured spectrum. If \( M \) is not free, Auslander–Bridger theory gives an exact sequence
	\[
	0 \to \Ext^1_R(\Tr M,M) \to M \otimes_R M^* \to \Hom_R(M,M) \to \Ext^2_R(\Tr M,M) \to 0.
	\]
	One shows that \( \Ext^2_R(\Tr M,M) \) has finite length. Moreover, there are embeddings (see \cite{ACS})
	\[
	\Ext^1_R(\Tr M,M) \hookrightarrow H^0_{\fm}(M \otimes_R M), \quad
	\Ext^2_R(\Tr M,M) \hookrightarrow H^1_{\fm}(M \otimes_R M).
	\]
	Since \( M \otimes_R M \) is reflexive, both local cohomology modules vanish. Hence \( M \otimes_R M^* \cong \Hom_R(M,M) \), and therefore
	\(
	H^1_{\fm}(M \otimes_R M^*) = 0.
	\)
	By \cite[Lemma 8.1]{ACS}, this forces \( M \) to be free.
\end{proof}\medskip
\section{Dimension of Ext Modules}

Recall that $\grade(I;M)$ is the length of the longest $M$-sequence contained in $I$. Also,
\(
\grade(M):= \grade(\Ann(M),R).
\)
We begin by recalling the following fact.

\begin{fact}\label{beger2}
	(See \cite[3.7]{Beder}.)  
	Let $A$ be a local ring and $M$ a finitely generated $A$-module of finite projective dimension. Assume that
	\(
	\grade(M) + \dim(M) = \dim(A).
	\)
	Let $d := \dim(A)$ and $r := \dim(M)$. Then
	\(
	\dim\big(\Ext^{d-r}_A(M,A)\big) = r.
	\)
\end{fact}

\begin{lemma}\label{45}
	Let $R$ be complete, and let $P \neq 0$ be a perfect prime ideal of height $t$ and finite projective dimension. Then
	\[
	\dim\!\big(\Ext^{t-1}_R(P,R)\big) =
	\begin{cases}
	d - t & \text{if } t \ge 2, \\
	d & \text{if } t = 1.
	\end{cases}
	\]
\end{lemma}

\begin{proof}
	By Fact \ref{Auslander}, $R$ is a domain. In particular, $R$ is catenary and equidimensional. Hence the grade conjecture holds, so Fact \ref{beger2} applies.
	Since $P$ is perfect,
	\[
	t = \dim(R) - \dim(R/P) = \grade(R/P) = \pd(R/P),
	\]
	where the first equality follows from \cite[Lemma 2, p.\ 250]{mat}. Fact \ref{beger2} now yields
	\[
	\dim\!\big(\Ext^{t}_R(R/P,R)\big) = \dim(R/P).
	\]
	If $t > 1$, then by shifting,
	\[
	\Ext^{t}_R(R/P,R) \cong \Ext^{t-1}_R(P,R),
	\]
	and the result follows.
	
	Now suppose $t=1$. From the short exact sequence
	\(
	0 \to P \to R \to R/P \to 0
	\)
	we obtain
	\[
	0 = \Hom(R/P,R) \lo R \lo \Hom(P,R),
	\]
	so $\dim(\Hom(P,R)) \ge \dim(R)$. The reverse inequality always holds, hence
	\(
	\dim(\Hom(P,R)) = \dim(R).
	\)
\end{proof}

\begin{lemma}\label{47}
	Let $(R,\mathfrak{m})$ be a regular local ring of dimension $4$ containing $\mathbb{Q}$, and let $M$ satisfy $(S_3)$. Then
	\(
	\dim\!\big(\Ext^{1}_R(M,R)\big) \le 2.
	\)
\end{lemma}

\begin{proof}
	Since $M$ satisfies $(S_3)$, it is torsion-free and reflexive. For any $\mathfrak{p} \in \Spec(R)$ with $\operatorname{ht}(\mathfrak{p}) \le 3$, the Auslander–Buchsbaum formula implies that $M_{\mathfrak{p}}$ is free.
	By a construction of Miller (see \cite[Theorem 2.17]{syz}), there exists a free module $F$ and an exact sequence
	\[
	0 \lo F \lo M \lo P \lo 0 \quad(\ast),
	\]
	where $P$ is a prime ideal of height $2$ and $R/P$ is normal. By Serre’s criterion, $R/P$ satisfies $(S_2)$. Since
	\[
	\dim(R/P) = \dim(R) - \operatorname{ht}(P) = 2,
	\]
	it follows that $R/P$ is Cohen–Macaulay, so $P$ is perfect. Thus we are in the situation of Lemma \ref{45}.
From the long exact sequence of Ext-modules induced by $(\ast)$, we have
	\[
	\Ext^{1}_R(P,R) \lo \Ext^{1}_R(M,R) \lo \Ext^{1}_R(F,R) = 0,
	\]
	we see that $\Ext^{1}_R(M,R)$ is a quotient of $\Ext^{1}_R(P,R)$. Hence
	\[
	\dim\!\big(\Ext^{1}_R(M,R)\big) \le \dim\!\big(\Ext^{1}_R(P,R)\big) \le 2.
	\]
\end{proof}
By the punctured spectrum we mean $\Spec(R)^\circ:=\Spec(R)\setminus\{\fm\}$. 
\begin{theorem}\label{47thm}
	Let $(R,\mathfrak{m})$ be a regular local ring of dimension $4$ containing $\mathbb{Q}$. Let $M$ satisfy $(S_3)$ and $N$ be torsion-free and locally free over punctured spectrum. Then
	\(
	\dim\!\big(\Ext^{1}_R(M,N)\big) \le 2.
	\)
\end{theorem}

\begin{proof}
	Let $\Omega(-)$ denote the first syzygy. Using \cite{AB}, we have an exact sequence
	\[
	\Tor_2(\Tr \Omega(M),N) \to \Ext^{1}(M,R)\otimes N \xrightarrow{g}
	\Ext^{1}(M,N) \xrightarrow{f}
	\Tor_1(\Tr \Omega(M),N) \to 0.
	\]
	Set $K := \ker(f)$. Since $\Tor_1(\Tr \Omega(M),N)$ has finite length,
	\[
	0 \lo K \lo \Ext^{1}(M,N) \lo \Tor_1(\Tr \Omega(M),N) \lo 0
	\]
	implies
	\[
	\dim\!\big(\Ext^{1}(M,N)\big) = \dim(K).
	\]
	Let $L := \operatorname{im}(g)$. Then $L = K$, and there is a surjection
	\[
	\Ext^{1}(M,R)\otimes N \twoheadrightarrow L\to 0.
	\]
	Hence
	\[
	\dim\!\big(\Ext^{1}(M,N)\big) = \dim(L) \le \dim\!\big(\Ext^{1}(M,R)\otimes N\big).
	\]
	Because $N$ is locally free and torsion-free,  \begin{align*}\Supp(\Ext^{1}(M,R) \otimes N) &=\Supp(\Ext^{1}(M,R) )\cap \Supp( N) \\ &=\Supp(\Ext^{1}(M,R) )\cap \Spec(R)\\ &=\Supp(\Ext^{1}(M,R) ), \end{align*} In other words,
	\(
	\dim(\Ext^{1}(M,R)\otimes N) = \dim(\Ext^{1}(M,R)).
	\)
	Therefore
	\[
	\dim\!\big(\Ext^{1}(M,N)\big) = \dim\!\big(\Ext^{1}(M,R)\big) \le 2,
	\]
	by Lemma \ref{47}.
\end{proof}

\begin{remark}
	In Theorem \ref{47thm}, instead of assuming $\mathbb{Q} \subseteq R$, one may assume that there exists an infinite field $\bar{k} \subseteq R$ of characteristic $p>0$ such that $\bar{k}$ is separable over $k$; see \cite[Theorem 2.17]{syz}.
\end{remark}

\begin{proposition}
	Let $R$ be a $d$-dimensional Gorenstein ring and $M$ a module satisfying $(S_2)$ but not $(S_3)$. Then
	\(
	\dim\!\big(\Ext^{d-2}_R(M, N)\big) = 0,
	\)
where  $N$  is torsion-free and locally free over punctured spectrum.  
\end{proposition}

\begin{proof}
In view of the following exact sequence: $$
\Tor_2^R(\Tr\Omega^{d-2}M,N)\rightarrow\Ext^{d-2}_R(M,R)\otimes_RN\rightarrow\Ext^{d-2}_R(M,N)\rightarrow\Tor_1^R(\Tr\Omega^{d-2}M,N)\rightarrow0,
$$and in the vein similar to  {Theorem} \ref{47thm}, the argument can be reduced to assume in addition that $N=R$. Now,	Grothendieck’s finiteness theorem (\cite[Section 9]{BSh}) gives
	\[
	f_{\mathfrak m}(M)
	= \inf\{\depth_{R_{\mathfrak p}}(M_{\mathfrak p}) + \dim(R/\mathfrak p) 
	: \mathfrak p \in \Supp(M)\setminus\{\mathfrak m\}\}.
	\]
	Since $M$ satisfies $(S_2)$ but not $(S_3)$, we have $f_{\mathfrak m}(M) > 2$. By definition,
	\[
	f_{\mathfrak m}(M)
	= \inf\{i : H^i_{\mathfrak m}(M) \text{ is not finitely generated}\},
	\]
	so $H^2_{\mathfrak m}(M)$ is finitely generated. As its support is contained in $\{\mathfrak m\}$, it has finite length.
	Let $(-)^\vee$ denote the Matlis dual. By local duality (\cite[Section 8]{BSh}),
	\(
	H^2_{\mathfrak m}(M) \cong \Ext^{d-2}_R(M, R)^\vee.
	\)
	Since Matlis duality preserves finite length, $\Ext^{d-2}_R(M,R)$ is artinian, hence
	\(
	\dim\!\big(\Ext^{d-2}_R(M,R)\big) = 0.
	\)
\end{proof}

\begin{fact}\label{art}
	If there exists a nonzero Artinian $R$-module of finite projective dimension, then $R$ is Cohen–Macaulay.
\end{fact}

\begin{fact}\label{dep}
	If $0 \neq A$ is an Artinian module, then
	\(
	\depth(R) = \inf\{ i : \Ext^i_R(A,R) \neq 0\}.
	\)
\end{fact}

By 
$\operatorname{codim}(M)$, we mean \( \inf\{\operatorname{ht}(\mathfrak p) : \mathfrak p \in \Supp(M)\}.
\)

\begin{proposition}
	Let $c := \operatorname{codim}(M)$ and suppose $\pd(M) < \infty$. The following  are valid.
	
	\begin{enumerate}
		\item[(i)] \(\dim\!\big(\Ext^{c}_R(M,R)\big) \ge \dim(R) - c.\)
		
		\item[(ii)] If $M$ is equidimensional, then
		\(\dim\!\big(\Ext^{c}_R(M,N)\big) = \dim(M),\) where  $N$  is torsion-free and locally free over punctured spectrum. 
	\end{enumerate}
\end{proposition}

\begin{proof}
	(i) Choose $\mathfrak p \in \Supp(M)$ with $\operatorname{ht}(\mathfrak p)=c$. Then $\mathfrak p \in \Min(\Supp(M))$, so $M_{\mathfrak p}$ is Artinian. Since $\pd(M_{\mathfrak p}) < \infty$, Fact \ref{art} implies $R_{\mathfrak p}$ is Cohen–Macaulay, hence
	\(
	\depth(R_{\mathfrak p}) = \dim(R_{\mathfrak p}) = c.
	\)
	By Fact \ref{dep}, $\Ext^c_{R_{\mathfrak p}}(M_{\mathfrak p}, R_{\mathfrak p}) \neq 0$, so
	$\mathfrak p \in \Supp(\Ext^c_R(M,R))$. Thus
	\[
	\dim\!\big(\Ext^c_R(M,R)\big) \ge \dim(R/\mathfrak p) = \dim(R) - c.
	\]
	
	(ii) In the vein similar to  {Theorem} \ref{47thm}, the proof easily reduced to the case that $N:=R$. If $M$ is equidimensional, every saturated chain from $\mathfrak p \in \Min(\Supp M)$ to $\mathfrak m$ has length $\dim(M)$. Hence $\dim(R) - c = \dim(M)$, and the result follows from (i).
\end{proof}

\begin{proposition}
	Let $(R,\mathfrak m)$ be local of dimension $d$, and let $M$ be a finitely generated module with $\pd(M) < \infty$. Then $\Ext^{d}_R(M,N)$ has finite length, where  $N$  is torsion-free and locally free. 
\end{proposition}

\begin{proof}The proof easily reduced to the case that $N:=R$ (for example, see the argument of {Theorem} \ref{47thm}).
	Assume $\Ext^{d}_R(M,R) \neq 0$. Then $d \le \pd(M) \le \depth(R) \le d$, so $\depth(R)=d$, i.e.\ $R$ is Cohen–Macaulay.
If $\dim(M)=0$ the result is clear. Otherwise choose $\mathfrak p \in \Supp(M)$ with $\mathfrak p \ne \mathfrak m$. Then
	\[
	\pd(M_{\mathfrak p}) \le \depth(R_{\mathfrak p}) = \operatorname{ht}(\mathfrak p) < d,
	\]
	so $\Ext^{d}_{R_{\mathfrak p}}(M_{\mathfrak p}, R_{\mathfrak p}) = 0$. Hence $\mathfrak p \notin \Supp(\Ext^d_R(M,R))$, and therefore
	\(
	\Supp(\Ext^d_R(M,R)) \subseteq \{\mathfrak m\}.
	\)
	Thus $\Ext^d_R(M,R)$ has finite length.
\end{proof}

\begin{proposition}
	Let $R$ be $d$-dimensional, and let $M$ be a $1$-dimensional module with $\pd(M) < \infty$. Then 
	\(
	\dim(R/\mathfrak p) + \operatorname{ht}(\mathfrak p) = \dim(R) 
	\) for every $\mathfrak p \in \Supp(\Ext^{d-1}_R(M,R))$.
\end{proposition}

\begin{proof}
	If $\Ext^{d-1}_R(M,R)=0$, there is nothing to prove. Let $\mathfrak p \in \Supp(\Ext^{d-1}_R(M,R))$.
First, assume that $\mathfrak p = \mathfrak m$. The the equality is trivial. Otherwise $\mathfrak p \in \Min(\Supp(M))$, since $\Supp(\Ext^{d-1}(M,R)) \subseteq \Supp(M)$ and $\dim(M)=1$. Thus $\dim(M_{\mathfrak p})=0$.
	Since $\pd(M_{\mathfrak p})<\infty$, Fact \ref{art} implies $R_{\mathfrak p}$ is Cohen–Macaulay. By Auslander–Buchsbaum,
	\[
	\pd(M_{\mathfrak p}) \le \depth(R_{\mathfrak p}) = \dim(R_{\mathfrak p}) = \operatorname{ht}(\mathfrak p).
	\]
	On the other hand,
	\[
	\Ext^{d-1}_{R_{\mathfrak p}}(M_{\mathfrak p}, R_{\mathfrak p}) \neq 0
	\;\Rightarrow\;
	\pd(M_{\mathfrak p}) \ge d-1.
	\]
	Hence
	\[
	\operatorname{ht}(\mathfrak p) \ge d-1 = d - \dim(R/\mathfrak p),
	\]
	so
	\(
	\dim(R/\mathfrak p) + \operatorname{ht}(\mathfrak p) \ge d.
	\)
	The reverse inequality always holds, giving equality.
\end{proof}We now pose the following:
\begin{problem}
	What is the annihilator of $\Ext^i_R(M,N)$?
\end{problem}\medskip
\section{Inequalities for Dimensions, Depth, and Grade}

The following is well known.

\begin{fact}\label{71}
	Let $\mathfrak{p} \in \Supp(M)$. Then
	\(
	\dim(M) \ge \dim(M_{\mathfrak{p}}) + \dim(R/\mathfrak{p}).
	\)
\end{fact}

\begin{proof}
	Set $t_1 = \dim(R/\mathfrak{p})$ and $t_2 = \dim(M_{\mathfrak{p}})$. Choose a chain of primes in $\Supp(M_{\mathfrak{p}})$:
	\[
	\mathfrak{q}_0 \subsetneq \mathfrak{q}_1 \subsetneq \cdots \subsetneq \mathfrak{q}_{t_2} \subseteq \mathfrak{p}.
	\]
	Also choose a chain in $\Supp(M)$ extending beyond $\mathfrak{p}$:
	\[
	\mathfrak{p} = \mathfrak{q}_{t_2} \subsetneq \mathfrak{q}_{t_2+1} \subsetneq \cdots \subsetneq \mathfrak{q}_{t_2+t_1} = \mathfrak{m}.
	\]
	Concatenating these chains gives a chain of length $t_1+t_2$ in $\Supp(M)$, proving the claim.
\end{proof}

\begin{fact}\label{psg}
	(Peskine–Szpiro, \cite[Lemma II.1.4.8]{PS}.) One has
	\(
	\depth(R) \le \grade(M) + \dim(M) \le \dim(R).
	\)
\end{fact}

By Cohen–Macaulay defect, we mean
\(
\CMdef(M) := \dim(M) - \depth(M).
\)

\begin{observation}\label{4.3}
	Let $\mathfrak{p} \in \Supp(M)$. Then:
	
	\begin{enumerate}
		\item[(i)] 
		\(
		\dim(R_{\mathfrak{p}}) - \dim(M_{\mathfrak{p}})
		\ge \dim(R) - \dim(M) - \CMdef(R).
		\)
		
		\item[(ii)] The inequality in (i) is sharp in general.
	\end{enumerate}
\end{observation}

\begin{proof}
	(i) Using Fact \ref{psg},
	\begin{align*}
	\dim(R_{\mathfrak{p}}) - \dim(M_{\mathfrak{p}})
	&\ge \grade(M_{\mathfrak{p}}) \\
	&\ge \grade(M) \\
	&\ge \depth(R) - \dim(M) \\
	&= \dim(R) - \dim(M) - \CMdef(R).
	\end{align*}
	Here we used in second inequality that $\Ext^i(M,R)=0$ implies $\Ext^i(M_{\mathfrak{p}},R_{\mathfrak{p}})=0$.
	
	(ii) Let $R = k[x,y,z]/((x)\cap(y,z))$ and $M = R/(y,z)$. Then
	$\dim(R)=2$ and $\dim(M)=1$. Localizing at $\mathfrak{p}=(y,z)$, we have
	$M_{\mathfrak{p}}=0$, so
	\[
	\dim(R_{\mathfrak{p}}) - \dim(M_{\mathfrak{p}}) = 0,
	\]
	while $\dim(R)-\dim(M)=1$.
\end{proof}

\begin{conjecture}\label{4.4}
	If $M$ has finite projective dimension, then for all $\mathfrak{p} \in \Supp(M)$,
	\[
	\dim(R_{\mathfrak{p}}) - \dim(M_{\mathfrak{p}})
	\ge \dim(R) - \dim(M).
	\]
\end{conjecture}

\begin{remark}
Recall that if the grade conjecture holds, then {Conjecture} \ref{4.4} is also true. Indeed, \[
\depth(R) - \dim( M) = \grade(M) \stackrel{(+)}\leq \grade(M_{\mathfrak{p}})  =\dim (R_{\mathfrak{p}}) - \dim (M_{\mathfrak{p}}),
\]where $(+)$  is in Observation \ref{4.3}(+).
\end{remark}

\begin{observation}
	Suppose $R$ is generalized Cohen–Macaulay. Then for any $\mathfrak{p} \in \Supp(M)$,
	\[
	\depth(R) - \dim (M )\leq \depth(R_{\mathfrak{p}}) - \depth(M_{\mathfrak{p}}).
	\]
\end{observation}

\begin{proof}
	We may assume $\mathfrak{p} \neq \mathfrak{m}$. Then
	\[
	\depth(R) - \dim (M) \stackrel{\ref{psg}}\leq \grade(M) \stackrel{(+)}\leq \grade(M_{\mathfrak{p}})\stackrel{\ref{psg}} \leq \dim (R_{\mathfrak{p}}) - \dim (M_{\mathfrak{p}}),
	\]where $(+)$  is in Observation \ref{4.3}(+).
	By   \cite[Exercise 9.5.7]{BSh}, $R_{\mathfrak{p}}$ is Cohen–Macaulay, so $\depth(R_{\mathfrak{p}}) = \dim (R_{\mathfrak{p}})$. Also, $\depth(M_{\mathfrak{p}}) \leq \dim( M_{\mathfrak{p}})$. Hence
	\(
	\dim (R_{\mathfrak{p}}) - \dim (M_{\mathfrak{p}}) \leq \depth(R_{\mathfrak{p}}) - \depth(M_{\mathfrak{p}}).
	\)
\end{proof}

\begin{proposition}\label{pdef}
One has  $\CMdef(M)\geq \CMdef(R)$ for any $M$ of finite projective dimension satisfying the grade conjecture.
\end{proposition}
\begin{proof}
We have\begin{align*}	\depth(R)+\CMdef(M)
&= \depth(R)+\dim(M)-\depth(M) \\
&=\pd(M)+\dim(M)\\
&\geq  \grade(M)+\dim(M)\\
&=\dim(R) \\
&=\depth(R)+\CMdef(R),
\end{align*}
and so $\CMdef(M)\geq \CMdef(R)$.
\end{proof}

\begin{proposition}
For any $\mathfrak{p} \in \Supp(M)$,
	\(\dim(M_{\mathfrak{p}}) - \depth(M_{\mathfrak{p}})
\leq  \dim M - \depth(M).
	\)
\end{proposition}
\begin{proof}Let $P\subseteq Q$.
In view of \cite[9.3.4]{BSh} we have 
\[\depth(M_{Q}) \leq  \depth(M_P)+\Ht(Q/P)
\quad(\ast).
\]
Set
$e:= \dim(R_{\mathfrak{p}})  -[\dim(R)-\dim(R/\fp)-\depth(M_{\mathfrak{p}})]$. Now, let $\mathfrak{p} \in \Supp(M)$. Then, we have

\begin{align*}	\dim(M_{\mathfrak{p}}) - \depth(M_{\mathfrak{p}})
	&=(\dim(M_{\mathfrak{p}}) - \dim(R_{\mathfrak{p}}))+[\dim(R)-\dim(R/\fp)-\depth(M_{\mathfrak{p}})-e]\\
	&=(\dim(M_{\mathfrak{p}}) - \dim(R_{\mathfrak{p}}))+[\dim(R)-\Ht(\fm/ \fp)-\depth(M_{\mathfrak{p}})]-e\\
	&\stackrel{(\ast)}\leq   \dim(M_{\mathfrak{p}})+[\dim(R)-\depth(M)]-e\\
	&=[\dim(M_{\mathfrak{p}})-\dim(R_{\mathfrak{p}})+\dim(R)-e] -\depth(M)\\
	&=[\dim(M_{\mathfrak{p}})+\dim(R/{\mathfrak{p}})]-\depth(M) \\
	&\stackrel{\ref{71}}\leq\dim (M )- \depth(M),
	\end{align*}
	and so $\CMdef(M)\geq \CMdef(M_{\mathfrak{p}})$.
\end{proof}
\begin{fact}
\label{beder}	Let $(R,\mathfrak{m})$ be complete   and $P\in\Spec(R)$ be of finite projective dimension. Then $\Ht(P)  = \grade(R/P)$.
\end{fact}

\begin{observation}\label{8}
	Let $(R,\mathfrak{m})$ be complete   and $P\in\Spec(R)$ be of finite projective dimension. Then  
	\(
\dim(R)-\dim(R/P)\leq \depth(R) - \depth(R/P).
	\)
\end{observation}

\begin{proof}
By {Fact}
\ref{Auslander}, $R$ is integral domain. In particular, $R$ is equi-dimensional and catenary. Then
	\begin{align*}
	\dim(R)-\dim(R/P)
	=  \Ht(P)  
	&\stackrel{\ref{beder}}= \grade(R/P) \leq\pd(R/P) \\
	&= \depth(R) - \depth(R/P).
	\end{align*}
\end{proof}

\begin{corollary}\label{8c}
	Let $(R,\mathfrak{m})$ be complete   and $P\in\Spec(R)$ be of finite projective dimension. Then  
	\(
\CMdef(R)\leq \CMdef( R/P).
	\)
	In particular, if $R/P$ is Cohen-Macaulay, then $R$ is as well.
\end{corollary}

\begin{proof}
This is immediate by Observation \ref {8}.
\end{proof}

\begin{fact}\label{ab}(Auslander-Buchsbaum \cite[Page 649]{ABU1}). One has
$\dim (R_{\mathfrak{p}}) - \depth(R_{\mathfrak{p}}) \leq \dim (R) - \depth(R)$.
\end{fact}

\begin{corollary}\label{70}
	Let $(R,\mathfrak{m})$ be complete   and $P\in\Spec(R)$ be of finite projective dimension. Then  
	\(
	\CMdef(R_P)\leq \CMdef( R/P).
	\)
	In particular, if $R/P$ is Cohen-Macaulay, then $R_P$ is as well.
\end{corollary}

\begin{proof}
	We have
	\begin{align*}
\CMdef(R_P)
= \dim(R_P) - \depth(R_P) 
	&\stackrel{\ref{ab}}\leq \dim(R  ) - \depth(R)\\
	&\stackrel{\ref{8}}\leq\dim(R/P) - \depth(R/P)  =\CMdef(R/P),
	\end{align*}as claimed.
\end{proof}

 \begin{remark}\label{54}Let $d := \dim (R)$ and $M$ be of finite projective dimension. If $\dim (M) = d-1$, then  
\(
\grade(M) + \dim (M )= d.
\)
\end{remark}
\begin{proof}
	Let $I := \Ann(M)$. If $I = 0$, then $\dim (M) = \dim (R)$, and the equality holds trivially. So assume $I \neq 0$. By   \cite[Lemma 2.28]{wol}  there exists $x \in I$ such that $x$ is $R$-regular. Hence
	\(
	\grade(M) = \grade(I, R) \geq 1.
	\)
	On the other hand,
	\(
	\grade(M) + \dim (M) \leq \dim (R).
	\)
	But $\dim( M) = d-1$, so $\grade(M) \leq 1$. Therefore $\grade(M) = 1$. In summary
	\(
	\grade(M) + \dim (M )= 1 + (d-1) = d.
	\)
\end{proof}
\begin{corollary}
	Let $d := \dim( R)$ and $M$ be of finite projective dimension. If $\dim (M )= d-1$, then  	\(
	\dim (\Ext^{1}_R(M, R)) = \dim (R) - 1.
	\)
\end{corollary}

\begin{proof}Note that
	$\grade(M)=\dim(R)-\dim(M)=1,$ because of
 Remark \ref{54}. It remains to apply  Fact \ref{beger2}.
\end{proof}

\begin{remark}\label{80}
	Let $(R,\mathfrak{m})$ be a d-dimensional UFD, $M$ a generalized Cohen–Macaulay module of dimension at least $d-2$, with $\pd(M) < \infty$. Then
$
	\grade(M) + \dim (M) = \dim( R). $
\end{remark}

\begin{proof}
	Let $I = \Ann(M)$ and recall from
	\cite{Beder} that 
		$\grade(M)=\operatorname{ht}(I)$.
Let $\mathfrak{p} \in \Min(I)$ so that $\operatorname{ht}(I) = \operatorname{ht}(\mathfrak{p})$. Note that $\mathfrak{p} \neq \mathfrak{m}$.
 By \cite[Exercise 9.5.7]{BSh}, we know that $\dim( M) = \dim (R/\mathfrak{q})$ for all $\mathfrak{q} \in \Ass(M) \setminus \{\mathfrak{m}\}$. Hence $\dim (M) = \dim (R/\mathfrak{p})$.
	Also, if $\Ht(\mathfrak{p})=1$ by the UFD assumption, $\mathfrak{p}$ should be principal, say $\mathfrak{p}=xR$. This implies that $\dim(R/\fp)=\dim(R/xR)=\dim (R)-1$, contradicting the fact that $\dim (M) = \dim( R/\mathfrak{p}) \geq d-2$. So, $\Ht(\fp)\geq 2$.
	Let us collect all things together and observe that\begin{align*}	\grade(M) + \dim (M)
	&= \operatorname{ht}( {I})+\dim(R/\mathfrak{p}) \\
	&=\operatorname{ht}(\mathfrak{p})+\dim(R/\mathfrak{p}) \\
	 &\geq 2+ (d-2)
	=\dim(R) \\
	&\stackrel{\ref{psg}}\geq\grade(M) + \dim (M).
	\end{align*}
 
In summary, we proved that $\grade(M) + \dim (M) = \dim (R).$	
\end{proof}

\begin{fact}\label{khi}
The following are well-known:
\begin{enumerate}
	\item If $\dim (N) \leq 1$ and $\dim (M) + \dim (N) < \dim (R)$, then $\Tor_i(M,N) = 0$ for all $i$.
	\item If $\chi(M,N) = 0$, then the grade equation holds for $M$.
\end{enumerate}\end{fact}
 Let us the above advanced results, to extend {Remark} \ref{80}: 
\begin{proposition}
	In fact, if $\dim (M) \geq \dim (R) - 2$ and $\pd(M) < \infty$, then
	\[
	\grade(M) + \dim( M) = \dim (R). \tag{*}
	\]
\end{proposition}

\begin{proof}
	Let $N$ be such that $\dim (M )+ \dim( N )< \dim( R)$. Since $\dim (M) \geq d-2$, we have $\dim (N) \leq 1$. By {Fact} \ref{khi}(1), $\Tor_i(M,N) = 0$ for all $i$. By {Fact} \ref{khi}(2), this implies $(*)$.
\end{proof}

\begin{corollary}
Let $R$ be 4-dimensional and $P\in\Spec(R)$ be of finite projective dimension. Then $
\grade(R/P) + \dim (R/P) = \dim (R). $
\end{corollary}

\begin{proposition}\label{9}
	Let $P\in \Spec(R)$ be of finite projective dimension. If $\Ht(P)=\dim(R)-1$, then $R$ is Cohen-Macaulay.	In particular, $R/P$ is perfect.
\end{proposition}

\begin{proof}
	Since $P\neq \fm$, $\depth(R/P)\neq 0$. Also, there is no prime ideal between $P$ and $\fm$. So, $$0<\depth(R/P)\leq \dim(R/P)=1\Rightarrow \depth(R/P)=1\quad(\ast)$$
	Then we have 	\begin{align*}
	\dim(R)-1
	&=  \Ht(P)  \\
	&\stackrel{\ref{beder}}= \grade(R/P) \leq\pd(R/P) \\
	&\stackrel{AB}= \depth(R) - \depth(R/P)\\
	&  \stackrel{(\ast)}=\depth(R) - 1
	\\
	&  \stackrel{(+)}\leq\dim(R) - 1.	\end{align*}This in particular, says that $(+)$ is an equality. In other words, $R$ is Cohen-Macaulay.	
\end{proof}
A ring is called almost Cohen-Macaulay
if $\depth(R)\geq \dim(R)-1$.
The following completes the proof of Theorem \ref{1.2}:
\begin{proposition}\label{10}
	Let $R$  be non-Cohen-Macaulay equipped with $P\in \Spec(R)$   of finite projective dimension. If $\Ht(P)=\dim(R)-2$, then

	\begin{enumerate}
		\item $R/P$ is perfect, 
		\item  	$R$ and $R/P$ are almost Cohen-Macaulay.
	\end{enumerate}
	
\end{proposition}

\begin{proof}
	By Fact~\ref{Auslander}, $R$ is an integral domain. From Corollary~\ref{8c} we have
	\(
	0 < \CMdef(R) \le \CMdef(R/P).
	\)
	Since $P$ is prime, $\depth(R/P) > 0$, and because $\dim(R/P)=2$, we obtain the following implication
	\[
	\CMdef(R/P) = 1 \quad \Longrightarrow \quad \depth(R/P) = 1. \tag{$\ast$}
	\]
Now, we compute the following:
	\begin{align*}
	\dim(R) - 2
	&= \dim(R) - \dim(R/P) \\
	&= \operatorname{ht}(P) \\
	&= \grade(R/P) \tag{\ref{beder}}\le \pd(R/P) \\
	&= \depth(R) - \depth(R/P) \tag{Auslander–Buchsbaum}\\
	&= \depth(R) - 1 \tag{$\ast$}\\
	&\le \dim(R) - 2.
	\end{align*}
	Again, all inequalities are equalities. Hence
	\[
	\depth(R) = \dim(R) - 1 \quad \text{and} \quad \depth(R/P) = \dim(R/P) - 1,
	\]
	so both $R$ and $R/P$ are almost Cohen–Macaulay. The equality
	\(
	\pd(R/P) = \grade(R/P)
	\)
	shows that $R/P$ is perfect.
\end{proof}
\medskip
\section{A question by Jorgensen}

This section {addresses the following question}

\begin{question}(Jorgensen \cite[Question 2.7]{Jo}).
	Suppose $\Ext^i(M,M) = 0$, $\pd(M) < \infty$, and $R$ is a complete intersection. Is $\pd(M) < i$?
\end{question}

The following result extends \cite[Proposition 2.5]{Jo} by Jorgensen, who proved {–over complete-intersection rings the following implication:–} {+over complete-intersection rings that}
$$\Ext^2(M,M)=0\Rightarrow \pd(M)<2\quad(+)$$

\begin{proposition}
Let $I$ be  a reflexive ideal
in a complete intersection ring $R$. If $\Ext^2(I,I) = 0$,   then $I$ is free.
\end{proposition}

\begin{proof}
	Recall that we may assume that $\pd(I) < \infty$.
	We argue by induction on $d:=\dim (R)$.
	Suppose first that $d\leq 2$. We apply the  Auslander–Buchsbaum formula and deduce that $I$ is free. 
Thus we may assume $d \geq 3$. By the inductive hypothesis, we may further assume that $I$ is locally free on the punctured spectrum.
	
	Recall that $I$ is reflexive. This implies that $\depth(I) \geq 2$. By \cite[Lemma 3.2]{ACS}, there is an embedding
	\[
	\Ext^1(I,I) \hookrightarrow H^2_{\mathfrak{m}}(\Hom(I,I)) 
	\quad(\ast)
	\]
	
	Let $r\in R$ and $x\in I$. The assignment
	$r \mapsto (x \mapsto rx)$
	defines a map $\varphi: R \to \Hom(I,I)$.  
	Clearly, $\varphi$ is injective since $I$ is torsion-free. Let $C := \Coker(\varphi)$.
	Since $I$ is locally free on the punctured spectrum, 
	$C_\mathfrak{p} = 0$ for all $\mathfrak{p} \neq \mathfrak{m}$,
	i.e., $C$ has finite length.	
	Then, by Grothendieck's vanishing theorem, $H^i_{\mathfrak{m}}(C) = 0$ for $i > 0$.
	From the exact sequence $0 \rightarrow R \rightarrow \Hom(I,I) \rightarrow C \rightarrow 0$,
	we deduce that
	\[
	0= H^2_{\mathfrak{m}}(R) \cong H^2_{\mathfrak{m}}(\Hom(I,I)),
	\]
	since $\depth(R)=d>2$. In view of $(\ast)$ we obtain
	$\Ext^1(I,I) = 0$.
	Since $\pd(I) \leq 1$ (see $(+)$ above), we conclude that $\pd(I) = 0$, because otherwise $\Ext^1(I,I) \neq 0$.
\end{proof}

Recall that a module $M$ is called rigid if $\Ext^1(M,M)=0$. The freeness of a rigid ideal over a 1-dimensional Gorenstein ring is a challenging problem in commutative algebra.

\begin{corollary}
	Let $R$ be a strongly normal ring of dimension at most $3$ and $I$ a reflexive ideal. Then $I$ is rigid.
\end{corollary}

\begin{proof}
	We may assume that $\dim(R)=3$. Strongly normal rings satisfy Serre's condition $(R_2)$. By this, $I$ is locally free on the punctured spectrum. Also, strongly normal rings satisfy Serre's condition $(S_3)$. Since $\dim(R)=3$, this means that $R$ is Cohen–Macaulay. Now we follow the previous argument to deduce that $\Ext^1(I,I) = 0$.
\end{proof}

\begin{fact}\label{rob}
	(See Roberts \cite[Theorem 2]{Ro2}). $\Ext^i_{\mathfrak{p}}(k, M) \neq 0$ iff $i \in [\depth(M), \id(M)]$.  
\end{fact}

By $ \mu_{j}(\mathfrak{m}, M)$ we mean $\dim_k(\Ext^j_{R}(k, M))$, where $k$ is the residue field.

\begin{theorem}
	Let $(R,\mathfrak{m})$ be a $d$-dimensional Gorenstein ring with $d>1$, let $\pd(M) < \infty$, and suppose $\Ext^{d-1}(M,M)=0$. If $M$ is quasi-Buchsbaum, then $\pd(M) \leq d-2$.
\end{theorem}

\begin{proof}
	Set  $\overline{M} := M/\Gamma_{\mathfrak{m}}(M)$.
	Suppose, for a contradiction, that $p:=\pd(M) \geq d-1$. Since
	$$\Ext^p(M,M) \neq 0 \Rightarrow \pd(M) \neq d-1,$$
and so $\pd(M) \geq d$. But,	the Auslander–Buchsbaum formula yields $\pd(M) \leq d$. Hence $\pd(M) = d$ and so $\depth(M) = 0$. In particular, $ \Gamma_{\mathfrak{m}}(M) \neq 0$.
	
	We may assume $\overline{M} \neq 0$. Indeed, if $\overline{M} = 0$, then $ M = \Gamma_{\mathfrak{m}}(M)$. Since $M$ is quasi-Buchsbaum, this implies $M=\Gamma_{\mathfrak{m}}(M)=\bigoplus k$. Consequently, $\Ext^{d-1}(k,k) = 0$, because it is a direct summand of $\Ext^{d-1}(M,M)=0$. But this is impossible, since it would imply
	$$\Tor^R_{d-1}(k,k)^\vee = \Ext^{d-1}(k,k^\vee) = \Ext^{d-1}(k,k),$$
	and hence $\pd(k) \leq d-1$, forcing $R$ to be regular of Krull dimension $\leq d-1$, a contradiction.
	
	Thus $\overline{M} \neq 0$. Since $R$ is Gorenstein, any module of finite projective dimension has finite injective dimension. Hence we may use Ischebeck's formula:
	\[
	\sup\{i \mid \Ext^i(\overline{M},M) \neq 0\} = \depth(R) - \depth(\overline{M})\leq d-1,
	\]
	since $\depth(\overline{M})>0$. Thus $\Ext^d(\overline{M},M) = 0$.
	From $0 \rightarrow \Gamma_{\mathfrak{m}}(M) \rightarrow M \rightarrow \overline{M} \rightarrow 0$, we obtain
	\[
	0= \Ext^{d-1}(M,M) \lo \Ext^{d-1}(\Gamma_{\mathfrak{m}}(M),M) \lo \Ext^d(\overline{M},M)=0.
	\]
But $\Gamma_{\mathfrak{m}}(M) = \bigoplus_X k$, since $M$ is quasi-Buchsbaum. As $\depth(M)=0$, we have $X\neq \emptyset$. Hence
	$$\Ext^{d-1}(\Gamma_{\mathfrak{m}}(M),M)=0 \Rightarrow \Ext^{d-1}(R/\mathfrak{m}, M) = 0.$$
	Recall from \cite[3.1.17]{BH} that $\id(M)=\depth(R)$. Using Fact \ref{rob}, we get
	$$\mu_{d-1}(\mathfrak{m},M) = 0 \Rightarrow d-1 \notin [\depth(M), \id(M)] = [0,d],$$
	a contradiction. This shows $\pd(M) \leq d-2$.
\end{proof}
\begin{fact}\label{depth}
	Suppose $\depth_R(R) >0$. Then $\depth_R(\mathfrak{m}) =1$.
\end{fact}

\begin{proof}
	We have $0\to \mathfrak{m}\to R\to k\to 0$. This yields
	$0=H^0_{\mathfrak{m}}(R)\to [k=H^0_{\mathfrak{m}}(k)]\to H^1_{\mathfrak{m}}(\mathfrak{m})$,
	hence $H^1_{\mathfrak{m}}(\mathfrak{m})\neq 0$. Therefore $\depth(\mathfrak{m}) =1$.
\end{proof}

\begin{lemma}\label{118}
	If $I$ is a radical ideal and $\Ext^j(I,I) = 0$ with $\pd(I) < \infty$, then
	\(
	j > \operatorname{ht}(I) - 1.
	\)
\end{lemma}

\begin{proof}
	We have $I  = \bigcap_{\mathfrak{q}\in\min(I)} \mathfrak{q}$. Take $\mathfrak{p} \in \operatorname{Assh}(I)$. 
	Then $\operatorname{ht}(I)=\operatorname{ht}(\mathfrak{p})$.
Recall that
	\[
	IR_\mathfrak{p} = [\bigcap_{\mathfrak{q}\in\min(I)\setminus\{\mathfrak{p}\}}  \mathfrak{q}R_\mathfrak{p}]  \cap  \mathfrak{p} R_\mathfrak{p} = R_\mathfrak{p} \cap \mathfrak{p} R_\mathfrak{p}
	=\mathfrak{p} R_\mathfrak{p}.
	\]
Then
	\begin{align*}
	0
	&= \Ext^j(I,I)_\mathfrak{p}   \\
	&= \Ext^j(IR_\mathfrak{p}, I R_\mathfrak{p})  \\
	&= \Ext^j(\mathfrak{p} R_\mathfrak{p}, \mathfrak{p} R_\mathfrak{p}) \\
	&= \Ext^{j+1}(R_\mathfrak{p}/\mathfrak{p} R_\mathfrak{p}, \mathfrak{p} R_\mathfrak{p}) .
	\end{align*}
Consequently, $\mu_{j+1}(\mathfrak{p} R_\mathfrak{p}, \mathfrak{p} R_\mathfrak{p}) = 0$.
	By Fact \ref{rob}, this implies that $\id(\mathfrak{p} R_\mathfrak{p})<\infty$. By Bass' conjecture, $R_\mathfrak{p}$ is Cohen–Macaulay. By \cite[3.1.17]{BH}, we have 
	\[
	\id(\mathfrak{p} R_\mathfrak{p})=\depth(R_\mathfrak{p})=\dim(R_\mathfrak{p})=\operatorname{ht}(\mathfrak{p})=\operatorname{ht}(I)\quad(+).
	\]
Using Roberts' result again, we obtain
	\[
	j+1 \notin [\depth(\mathfrak{p} R_\mathfrak{p}), \id(\mathfrak{p} R_\mathfrak{p})]
	\stackrel{\ref{depth}}{=} [1, \id(\mathfrak{p} R_\mathfrak{p})]
	\stackrel{(+)}{=} [1, \operatorname{ht}(I)].
	\]
	Thus $j+1 > \operatorname{ht}(I)$, i.e., $j > \operatorname{ht}(I)-1$.
\end{proof}

\begin{theorem}
	Let $R$ be Cohen–Macaulay and $I \subset R$ a radical ideal with $\pd(I) < \infty$ such that $R/I$ is Cohen–Macaulay. If $\Ext^i(I,I) = 0$, then $\pd(I) < i$.
\end{theorem}

\begin{proof}
	By Lemma \ref{118} 
we know $	\operatorname{ht}(I) - 1  < i$. Now, we combine this with the Auslander–Buchsbaum formula, and deduce
	\begin{align*}
	\pd(I)
	&= \pd(R/I) - 1 \\
	&= [\depth(R)  - \depth(R /I)] - 1 \\
	&= [ \dim (R) - \dim ( R/I) ] - 1 \\
	&= \operatorname{ht}(I) - 1 < i,
	\end{align*}
	as claimed.
\end{proof}

The following completes the proof of Theorem \ref{1.4}:

\begin{proposition}
	A finite-length counterexample to Question 5.1 over a complete intersection would require dimension at least \(6\).
\end{proposition}

\begin{proof}
	Suppose \(M\ne 0\) has finite length and finite projective dimension. The \(n=3\) implication holds when \(R\) is Gorenstein with \(\dim R\le 4\), and when \(R\) is a complete intersection with \(\dim R\le 5\).
	Here is the verification. Put \(d:=\dim R\), \(E^i:=\Ext^i_R(M,M)\), and let \(D\) denote Matlis duality. Auslander--Buchsbaum gives \(\pd_R M=d\). Gorenstein duality gives
$
		E^{d-i}\cong D(E^i).
$

	Assume for a moment that \((R,\mathfrak m)\) is a \(d\)-dimensional Gorenstein local ring and \(M\) is a finitely generated \(R\)-module with \(\pd_R M=d\). Write
$
	D(-)=\Hom_R(-,E_R(k)).
$
Let \(F_\bullet\to M\) be a finite free resolution of length \(d\):
	\[
	0\longrightarrow F_d\longrightarrow\cdots
	\longrightarrow F_0\longrightarrow M\longrightarrow0.
	\]
	
	Applying \(\Hom_R(-,R)\) gives the dual complex
	\[
	0\longrightarrow F_0^*\longrightarrow\cdots
	\longrightarrow F_d^*\longrightarrow0,
	\qquad F_j^*=\Hom_R(F_j,R).
	\]
	
	Because \(R\) is Gorenstein and \(\pd_R M=d=\dim R\), Auslander--Buchsbaum implies
$
	\depth_R M=0.
$
 Moreover, local duality gives
$
	\Ext_R^j(M,R)=0$ $j\ne d$
	and
$
	\Ext_R^d(M,R)\cong D(H^0_\fm(M))=  DM.
$
Thus the dual complex, when its degrees are reversed, is a free resolution of \(DM\). More precisely, the resolution is
	\[
	0\longrightarrow F_0^*
	\longrightarrow\cdots\longrightarrow F_d^*
	\longrightarrow DM\longrightarrow0,
	\]
	where the terms are indexed in the reverse order.
	Tensoring this resolution with \(M\), we obtain
$
	\Tor_i^R(DM,M)
	\cong H_i(F_\bullet^*\otimes_R M),
$
identifying \(F_j^*\otimes_R M\) with
	\(\Hom_R(F_j,M)\), gives
	${\Tor_i^R(DM,M)
		\cong \Ext_R^{d-i}(M,M).}
	$
Take the free resolution \(F_\bullet\to M\). Then
$
	E^i=H^i\bigl(\Hom_R(F_\bullet,M)\bigr).
$
	Applying \(D\), which is exact, gives
$	D(E^i)
	\cong H_i\left(
	D\bigl(\Hom_R(F_\bullet,M)\bigr)
	\right).
$ For each finite free module \(F_j\), tensor--Hom adjunction gives
	\[
	D\bigl(\Hom_R(F_j,M)\bigr)
	\cong F_j\otimes_R D(M).
	\]
	
	Therefore,
$
	D(E^i)
	\cong H_i(F_\bullet\otimes_R DM)
	=\Tor_i^R(M,DM).
$	Since  Tor is symmetric, so
	${D(E^i)\cong\Tor_i^R(DM,M).}$
Combining the two isomorphisms yields
	${E^{d-i}\cong D(E^i).}$
	For \(d=4\), the assumption \(E^3=0\) therefore implies \(E^1=0\). On the other hand,
	\[
	\sum_{i=0}^{4}(-1)^i\ell(E^i)
	=\ell(M)\sum_{i=0}^{4}(-1)^i\beta_i(M)=0.
	\]
	The last equality follows by localizing the free resolution at a minimal prime, where \(M\) vanishes. Hence
	\[
	0=\ell(E^0)+\ell(E^2)+\ell(E^4)>0,
	\]
	a contradiction.
	
	For \(d=5\) and \(R\) a complete intersection, $
	E^{d-i}\cong D(E^i) 
	$ turns \(E^3=0\) into \(E^2=0\). Jorgensen's theorem (see above) then gives \(\pd_R M\le 1\), contradicting \(\pd_R M=5\).
This completes the desired claim.
\end{proof}

\medskip
\section{Applications to $\Gdim(-)$}
Recall that $\Gdim_R(-)$ means G-dimension.
Our aim in this section is to address and study Conjecture \ref{con}. In this regard, we present the following results:

\begin{enumerate}
	\item[$\bullet$]	If $\Gdim(M) < \infty$, then 
	\( 
	\Gdim(M) = \sup\{ i \mid \Ext^i(M,R) \neq 0 \}.
	\) 
	\item[$\bullet$] 	Let $P \in \Spec(R)$ with $\Gdim(R/P^{(n)}) < \infty$ for some $n$. Then 
	$\grade(P^{(n)}; R) =  \operatorname{ht}(P^{(n)})= \operatorname{ht}(P)$.
	\item[$\bullet$]  If $R$ is analytically unramified and $\Gdim_R(\overline{R}) < \infty$, then $R$ is quasi-normal.
\end{enumerate}

Also, we present additional results relating these together.
Suppose now that $(R,\mathfrak{m})$ is catenary and equidimensional and let $P \in \Spec(R)$ be such that $\Gdim(R/P) < \infty$. Consequently, we will show that
\(
\grade(R/P) + \dim (R/P) = \dim (R).
\) We close this section  with Remark~\ref{disc:hypersurface-question}.
We will use the following standard facts about $\Gdim_R(-)$ from \cite{AB}.

\begin{fact}\label{gd}
	If $\Gdim_R(k)<\infty$, then $R$ is Gorenstein.
\end{fact}

\begin{fact}\label{gd2}
	If $\Gdim_R(M)<\infty$, then $\Gdim_R(M)=\depth(R)-\depth(M)$.
\end{fact}

\begin{fact}\label{grs}
	For a finitely generated module $M$ over a  ring $R$, recall from many classical works including Grothendieck that
	\[
	\grade(M) = \inf\{ \depth(R_{\mathfrak{p}}) \mid \mathfrak{p} \in \Supp(M) \} \qquad (\ast)
	\]
\end{fact}

\begin{question}\label{6.1}
	Is $\grade(M)$ equal to $\inf\{ \depth(R_{\mathfrak{p}}) \mid \mathfrak{p} \in \Min(M) \}$?
\end{question}

\begin{proposition}
	Question \ref{6.1} is true provided $\pd(M) < \infty$.
\end{proposition}

\begin{proof}
By definition, and in view of $(\ast)$ from Fact \ref{grs}, there exists $Q \in \Supp(M)$ such that $\grade(M) = \depth(R_Q)$.  
Let us choose $P \subseteq Q$ in $\Supp(M)$ minimal with respect to inclusion. This gives us the following:
\begin{enumerate}
	\item $\grade(M) = \depth(R_Q)$, 
	\item $\dim(M_P)=0$.
\end{enumerate}
In the sequel, by $(AB)$ we mean the Auslander--Buchsbaum formula. Then
	\begin{align*}	 
	\grade(M)
	&\stackrel{(1)}= \depth(R_Q) \\
	&\stackrel{AB}= \depth(M_Q) + \pd(M_Q)\\
	&\geq \pd_{R_Q}(M_Q) \\
	&\geq \pd_{R_P}(M_P)\\
	&= \depth(R_P) - \depth(M_P) \\
	&\geq \depth(R_P) - \dim(M_P)\\
	&\stackrel{(2)}= \depth(R_P)\\
	&\geq\inf\{\depth(R_{\mathfrak{p}}) \mid \mathfrak{p} \in \Supp(M)\}\\
	&\stackrel{(\ast)}=\grade(M),
	\end{align*}as desired claimed.
\end{proof}

\begin{remark}\label{98}
	In fact, if $\Gdim(M) < \infty$, the same conclusion holds. Indeed,
	\[
	\Gdim(M) = \sup\{ i \mid \Ext^i(M,R) \neq 0 \}
	\]
	when finite, together with the Auslander–Bridger formula (Fact \ref{gd2}).
\end{remark}

\begin{proposition}\label{99a}
	Let $P \in \Spec(R)$ be such that $\Gdim(R/P) < \infty$. Then
	\[
	\grade(R/P) = \operatorname{ht}(P) = \operatorname{ht}(\Ann(R/P)).
	\]
\end{proposition}

\begin{proof}
	By Remark \ref{98}, $\grade(R/P) = \depth(R_Q)$ for some $Q \in \Min(R/P)=\{P\}$. Hence $Q=P$.  
	Since $\dim(R_P) = \operatorname{ht}(P)$, it suffices to show $\depth(R_P)=\dim(R_P)$.
Because $\Gdim(R/P) < \infty$, localization yields $\Gdim(R_P/PR_P) < \infty$. By Fact \ref{gd}, $R_P$ is Gorenstein, hence Cohen–Macaulay. Therefore
	\[
	\grade(R/P)= \depth(R_P) = \dim(R_P) = \operatorname{ht}(P).
	\]This is what we want to prove
\end{proof}

\begin{corollary}\label{99}
	Let $P \in \Spec(R)$ with $\Gdim(R/P^ n) < \infty$ for some $n$. Then $\grade(P^{n}; R) = \operatorname{ht}(P)$.
\end{corollary}

\begin{proof}
	We may assume $P\neq 0$. Localization gives $\Gdim(R_P/P^ nR_P) < \infty$. By \cite[Proposition 1.6]{CT}, $R_P$ is Gorenstein. Now apply Proposition \ref{99a}.
\end{proof}

By definition, $P^{(n)}:=P^{n} R_P\cap R$ is the $n$th symbolic power.

\begin{corollary}\label{100}
	Let $P \in \Spec(R)$ with $\Gdim(R/P^{(n)}) < \infty$ for some $n$. Then $\grade(P^{(n)}; R) = \operatorname{ht}(P)$.
\end{corollary}

\begin{proof}
	Since $P^{(n)} R_P =(P^{n} R_P\cap R) R_P= P^n R_P$, the claim follows as in Corollary \ref{99}.
\end{proof}

We conjecture the following.

\begin{conjecture}\label{con}
	If $\Gdim(M) < \infty$, then
	\(
	\grade(M) = \operatorname{ht}(\Ann(M)).
	\)
\end{conjecture}

There is a connection with the following well-known conjecture.

\begin{conjecture}\label{f}	
	If $\Gdim(M) < \infty$ and $M$ has finite length, then $R$ is Gorenstein.
\end{conjecture}

\begin{remark}
	Conjecture \ref{f} implies Conjecture \ref{con}.
\end{remark}

\begin{proof}
	Adapting the proof of Proposition \ref{99a}, let $P \in \Min(M)$. Then $\depth(R_P) = \grade(M)$. Since $M_P \neq 0$ has finite length, $R_P$ is Cohen–Macaulay. This suffices to deduce Conjecture \ref{con}.
\end{proof}

\begin{corollary}
	Let $(R,\mathfrak{m})$ be catenary and equidimensional. Let $P \in \Spec(R)$ be such that $\Gdim(R/P) < \infty$. Then
	\(
	\grade(R/P) + \dim (R/P )= \dim (R).
	\)
\end{corollary}

\begin{proof}
	We know that $\operatorname{ht}(P) + \dim (R/P) = \dim( R)$ (see \cite[Page 250]{mat}). From Proposition \ref{99a}, $\grade(R/P) = \operatorname{ht}(P)$. The result follows.
\end{proof}

\begin{observation}
	If $\Gdim(M) < \infty$ and $\ell(M) < \infty$, then $M$ is quasi-perfect. That is,
	\[
	\inf\{ i \mid \Ext^i_R(M,R) \neq 0 \} 
	= 
	\sup\{ i \mid \Ext^i_R(M,R) \neq 0 \}.
	\]
\end{observation}

\begin{proof}
	Since $\ell(M) < \infty$, we have $\depth(M) = \dim(M)= 0$. By  Auslander–Bridger formula (Fact \ref{gd2}),
	\[
	\depth(R) = \depth(M)+\Gdim(M)= \Gdim(M)=\sup\{ i \mid \Ext^i(M,R) \neq 0 \}.
	\]
	Also, $\Ann(M)$ is $\mathfrak{m}$-primary, so a maximal regular sequence in $\Ann(M)$ has length $\depth(R)$. Hence
	\[
	\depth(R) = \grade(M) = \inf\{ i \mid \Ext^i(M,R) \neq 0 \}.
	\]
	Thus $M$ is quasi-perfect.
\end{proof}

By $\overline{R}$ we mean the integral closure of $R$ in the total fraction ring of $R$. 

\begin{proposition}\label{116}
	If $R$ is analytically unramified and $\pd_R(\overline{R}) < \infty$, then $R$ is normal.
\end{proposition}

\begin{proof}
	Let $\mathfrak{p} \in \Spec(R)$. Since being analytically unramified is a local property \cite[Proposition 9.1.4]{HS}, $R_\mathfrak{p}$ is analytically unramified.  
	Suppose $\operatorname{ht}(\mathfrak{p}) = 1$. Then $R_\mathfrak{p} \to \overline{R_\mathfrak{p}}$ is integral. Since $\pd_R(\overline{R}) < \infty$, we get $\pd_{R_\mathfrak{p}}(\overline{R_\mathfrak{p}}) < \infty$. As $\dim (R_\mathfrak{p}) = 1$, the Auslander--Buchsbaum formula implies $\overline{R_\mathfrak{p}}$ is free over $R_\mathfrak{p}$. Hence $\overline{R_\mathfrak{p}}$ is $1$-dimensional and normal.
Let $Q \in \Min(\overline{\mathfrak{p}R_\mathfrak{p}})$. The composition
	\[
	R_\mathfrak{p} \to \overline{R_\mathfrak{p}} \to (\overline{R_\mathfrak{p}})_Q
	\]
	is faithfully flat, and $(\overline{R_\mathfrak{p}})_Q$ is regular. By \cite[Theorem 2.2.12]{BH}, $R$ satisfies $(R_1)$.
Now let $\operatorname{ht}(\mathfrak{p}) \geq 2$. Since $\overline{R_\mathfrak{p}}$ is finite over $R_\mathfrak{p}$, and by \cite[1.2.26(b)]{BH}, we have
	\[
	\depth_{R_\mathfrak{p}}(\overline{R_\mathfrak{p}})
	=
	\depth_{\overline{R_\mathfrak{p}}} (\overline{R_\mathfrak{p}}).
	\]
	As $\overline{R_\mathfrak{p}}$ is normal, it satisfies $(S_2)$, so its depth, as an $\overline{R_\mathfrak{p}}$-module, is at least $2$. The Auslander--Buchsbaum formula gives
	\[
	\depth(R_\mathfrak{p}) 
	= \pd_{R_\mathfrak{p}}(\overline{R_\mathfrak{p}}) + \depth_{R_\mathfrak{p}}(\overline{R_\mathfrak{p}})
	= \pd_{R_\mathfrak{p}}(\overline{R_\mathfrak{p}}) + \depth_{\overline{R_\mathfrak{p}}}(\overline{R_\mathfrak{p}})
	\ge 2.
	\]
	Thus $R$ satisfies $(S_2)$. By Serre's criterion, $R$ is normal.
\end{proof}

In the same vein, we observe that:

\begin{proposition}\label{116q}
	If $R$ is analytically unramified and $\Gdim_R(\overline{R}) < \infty$, then $R$ is quasi-normal.
\end{proposition}

\begin{remark}\label{disc:hypersurface-question}
	\begin{enumerate}
		\item[(i)]	A recent paper \cite{C2} asked whether every local hypersurface domain $R$ satisfies the following property for any positive integer $c$: 
		\begin{quote}
			If $M$ and $M\otimes_R \Ext^{c-1}_R(M,R)$ are Cohen--Macaulay $R$-modules, both of grade $c-1$, then does $M$ have finite projective dimension?
		\end{quote}
		
		\item[(ii)] Let $R$ be a local hypersurface domain. If the condition in Remark~\ref{disc:hypersurface-question}(i) holds for $c := \dim(R)+1$, then $R$ is regular. In particular, any singular hypersurface domain fails to satisfy the condition.
	\end{enumerate}
\end{remark}

\begin{proof}
	Take $M := R/\mathfrak{m}$. Since $R$ is Cohen--Macaulay, $M$ has grade $\depth(R)=\dim(R) = c-1$. Also, $M$ is Cohen--Macaulay because it is Artinian.
	
	Moreover, $\Ext^{c-1}_R(M,R) \cong k$ by local duality (or via the Koszul complex), so
	\[
	M\otimes_R \Ext^{c-1}_R(M,R) \cong k \otimes_R k \cong k,
	\]
	which is again Cohen--Macaulay of grade $c-1$. The stated condition would then force $M$ to have finite projective dimension, which by the Auslander--Buchsbaum--Serre theorem implies that $R$ is regular. Hence any singular hypersurface provides a counterexample.
\end{proof}

\medskip

\section{When is $\pd_R(M\otimes_R N)=1$?}

The following question was raised in \cite[Question 3.3]{celp}.

\begin{question}
	Do there exist nonfree torsion-free modules $M$ and $N$ over a local ring $R$ with $\depth(R)=2$ such that
	\(
	\pd_R(M\otimes_R N)=1?
	\)
\end{question}

\subsection{Working over hypersurfaces}

Let $R$ be a hypersurface, that is, $R=S/(f)$ where $(S,\mathfrak n)$ is regular and $0\ne f\in\mathfrak n$.
If $R$ has singularity, it means that $f\in\mathfrak n^2$.
\begin{proposition}
	Let $R$ be a hypersurface of dimension $d>2$, and let $M,N$ be modules that are locally free on the punctured spectrum. Then
	\(
	\pd_R(M\otimes_R N)\ne 1.
	\)
\end{proposition}

\begin{proof}
	Suppose, for a contradiction, that $\pd_R(M\otimes_R N)=1$. By the Auslander–Buchsbaum formula,
	\[
	\depth_R(M\otimes_R N)=d-1\ge2.
	\]
	Thus $M\otimes_R N$ satisfies Serre's $(S_2)$ condition, and hence is reflexive by \cite[3.6]{syz}. The second rigidity theorem of Huneke–Wiegand \cite[2.7]{tensor} then gives $\Tor_i^R(M,N)=0$ for all $i>0$, contradicting \cite[1.5]{celp}.
\end{proof}

Recall that $R$ has a \emph{simple singularity} if there are only finitely many ideals $I\subseteq S$ such that $f\in I^2$.

\begin{proposition}
	Let $R$ be a $2$-dimensional hypersurface with a simple singularity, and let $M,N$ be nonfree torsion-free modules. Then
	\(
	\pd_R(M\otimes_R N)\ne 1.
	\)
\end{proposition}

\begin{proof}
	Assume $\pd_R(M\otimes_R N)=1$. We first show that $M$ and $N$ are locally free on the punctured spectrum.
	Let $\mathfrak p$ be a height-one prime and set $A=R_{\mathfrak p}$. Then
	\[
	\pd_A(M_{\mathfrak p}\otimes_A N_{\mathfrak p})\le1.
	\]
	If this projective dimension is $0$, there is nothing to prove. Otherwise, suppose it equals $1$. Since $A$ is a hypersurface of dimension one, every torsion-free $A$-module is totally reflexive by the Auslander–Bridger formula. Hence
	\[
	\Ext^1_A(M_{\mathfrak p},\Omega(M_{\mathfrak p}\otimes_A N_{\mathfrak p}))=0.
	\]
	By \cite[3.4]{celp}, this implies that $M_{\mathfrak p}$ is free. The same argument shows $N_{\mathfrak p}$ is free.
	
	Thus $M$ and $N$ are locally free on the punctured spectrum. By \cite[A.12]{celp}, this implies $\pd_R(M)$ and $\pd_R(N)$ are finite, contradicting \cite[3.2]{celp}.
\end{proof}
Recall that $(-)^T$ denotes the trace of a matrix.
\begin{proposition}
	Suppose the presentation matrix of a nonfree torsionless module $M$ is symmetric. Then for every nonzero module $N$,
	\(
	\pd_R(M\otimes_R N)\ne 1.
	\)
\end{proposition}

\begin{proof}
	Let $N\ne0$. There is a presentation
	\(
	R^n \xrightarrow{A} R^n \to M \to 0
	\)
	with $A^T=A$. Applying $\Hom_R(-,R)$ gives
	\[
	(R^n)^* \xrightarrow{A^T} (R^n)^* \lo \Tr(M) \lo 0.
	\]
	Identifying $(R^n)^*\cong R^n$, we obtain an isomorphism $M\cong \Tr(M)$ by the Five Lemma. Since $M$ is torsionless,
	\(
	\Ext^1_R(\Tr(M),R)=0.
	\) If $\pd_R(M\otimes_R N)=1$, then
	\[
	\Ext^1_R(M,\Omega(M\otimes_R N))=0.
	\]
	By \cite[Lemma 3.4]{celp}, this forces $M$ to be free, a contradiction.
\end{proof}

\begin{remark}
	The torsionless assumption is essential. Let $x$ be a regular element of a ring $R$ with positive depth and set $M=R/xR$. Then
	\(
	0\to R \xrightarrow{x} R \to M \to 0
	\)
	is a free resolution with a symmetric presentation matrix. For any nonzero free module $N$, one has $\pd_R(M\otimes_R N)=1$.
\end{remark}
Let us give an example.
\begin{example}
	Let $S = k[x, y, z, w]$, $f = x^{2} + y^{2} + z^{2} + w^{2}$, and set $R = S/(f)$, where $\operatorname{char}(k) \neq 2$.
	Define two $2 \times 2$ matrices
	\[
	\varphi_{xy} = \begin{pmatrix}
	x & y \\
	-y & x
	\end{pmatrix}, \qquad
	\varphi_{zw} = \begin{pmatrix}
	z & w \\
	-w & z
	\end{pmatrix}.
	\]
	Let
	\[
	\varphi := \begin{pmatrix}
	\varphi_{xy} & 0 \\
	0 & \varphi_{zw}
	\end{pmatrix},
	\qquad
	\psi := \varphi^{\mathrm{T}}.
	\]
	A direct computation shows:
	\[
	\varphi_{xy}\varphi_{xy}^{\mathrm{T}} = (x^{2} + y^{2})I_{2}, \quad
	\varphi_{zw}\varphi_{zw}^{\mathrm{T}} = (z^{2} + w^{2})I_{2},
	\]
	hence
	\[
	\varphi\psi = (x^{2}+y^{2}+z^{2}+w^{2})I_{4} = f I_{4},
	\qquad
	\psi\varphi = f I_{4}.
	\]
	Thus $(\varphi, \psi)$ is a matrix factorization of $f$.
	Over $R$, we obtain a $2$-periodic free resolution of $M := \operatorname{coker}(\varphi)$:
	\[
	\cdots \xrightarrow{\psi} R^{4} \xrightarrow{\varphi} R^{4} \xrightarrow{\psi} R^{4} \xrightarrow{\varphi} M \longrightarrow 0.
	\]
	
	Now, define the block-diagonal symmetric matrix $B := \operatorname{diag}(B_{2},B_{2})$, where
	\[
	B_{2} := \begin{pmatrix}
	1 & 0 \\
	0 & -1
	\end{pmatrix}.
	\]
	A straightforward calculation shows that $\varphi^{\mathrm{T}}B = B\varphi$.
	Since $B$ is invertible and symmetric, it induces an isomorphism
	\[
	M \stackrel{\cong}{\longrightarrow} M^{*} := \operatorname{Hom}_{R}(M,R).
	\]
	Thus $M$ is self-dual.
	If $k$ contains an element $\imath$ such that $\imath^{2} = -1$ (for example $k:=\mathbb{C}$  ), we may set $P_{2} = \operatorname{diag}(1,\imath)$ and $P = \operatorname{diag}(P_{2},P_{2})$.
	Then $P^{\mathrm{T}}BP = I_{4}$, and $\widetilde{\varphi} := P^{-1}\varphi P$ is symmetric:
	\[
	\widetilde{\varphi} =
	\begin{pmatrix}
	x & \imath y & 0 & 0 \\
	\imath y & x   & 0 & 0 \\
	0 & 0 & z & \imath w \\
	0 & 0 & \imath w & z
	\end{pmatrix},
	\]
	and it still satisfies $\widetilde{\varphi}^{\,2}=f I_{4}$ in the matrix factorization.
\end{example}

\subsection{Working with bimodules}

Let $F:R\to R$ denote the Frobenius endomorphism. The $n$th iterate defines a new $R$-module structure on the set $R$, denoted ${}^{{F^n}}\!R$, where
\[
a\cdot b = a^{p^n}b \quad (a,b\in R).
\]
For an ideal $I$, we write $I^{[p]}=\langle x^p : x\in I\rangle$.

\begin{example}
	There exist modules $M,N$ of infinite projective dimension over the ring
	\(
	R=\overline{\mathbb F}_p[[X,Y,Z]]/(X^p),
	\)
	which has depth $2$, such that $\pd_R(M\otimes_R N)=1$. Moreover, $M$ is maximal Cohen–Macaulay.
\end{example}

\begin{proof}
	Let $N=R/(x,y)$ and $M={}^{F^n}\!R$. Then
	\[
	M\otimes_R N \cong R/(x,y)^{[p]}=R/(x^p,y^p)=R/(y^p)
	\]
	by \cite[8.2.2(b)]{BH}. Since $y^p$ is $R$-regular,
	\(
	\pd_R(M\otimes_R N)=1.
	\) Because $R$ is not regular, Kunz's theorem \cite[8.2.8]{BH} implies $\pd_R(M)=\infty$. If $\pd_R(N)<\infty$, then $\pd_R(R/xR)<\infty$, which would force $(x)$ to contain a regular element, impossible since $x^p=0$. Hence $\pd_R(N)=\infty$.
\end{proof}
\begin{observation}
	Let $R$ be a one-dimensional domain of prime characteristic $p>0$ that is not weakly normal. Then there exist finitely generated $R$-modules $M$ and $N$ such that
	\[
	\pd_R(M\otimes_R N)=1 \quad \text{while} \quad \pd_R(M)=\pd_R(N)=\infty.
	\]
\end{observation}

\begin{proof}
	Since $R$ is not weakly normal, there exists a principal ideal $I=xR$ that is not Frobenius closed. Hence there is an element $y\in R$ such that
	\[
	y^p \in I^{[p]}=(x^p) \quad \text{but} \quad y\notin I. \tag{$\ast$}
	\]
	
	Set
	\[
	N:=R/(x,y), \qquad M:={}^{F^n}\!R
	\]
	for some $n\ge 1$, where ${}^{F^n}\!R$ denotes $R$ viewed as an $R$-module via the $n$-th Frobenius.
	
	By \cite[8.2.2(b)]{BH},
	\(
	M\otimes_R N \cong R/(x,y)^{[p^n]} = R/(x^{p^n},y^{p^n}).
	\)
	Since $y^p \in (x^p)$, we have $y^{p^n}\in (x^{p^n})$, and therefore
	\(
	(x^{p^n},y^{p^n})=(x^{p^n}).
	\)
	Thus
	\(
	M\otimes_R N \cong R/(x^{p^n}).
	\)
	Because $R$ is a one-dimensional domain and $x\neq 0$, the element $x^{p^n}$ is a nonzerodivisor. Hence
	\[
	\pd_R(M\otimes_R N)=\pd_R\big(R/(x^{p^n})\big)=1.
	\]
	
	Since $R$ is not regular, Kunz’s theorem implies that Frobenius is not flat, and therefore
	\(
	\pd_R(M)=\pd_R({}^{F^n}\!R)=\infty
	\)
	(see \cite[8.2.8]{BH}).
	Now suppose, toward a contradiction, that $\pd_R(N)<\infty$. Then $\pd_R(R/(x,y))<\infty$. Consider the short exact sequence
	\(
	0 \longrightarrow (x,y) \longrightarrow R \longrightarrow R/(x,y) \longrightarrow 0.
	\)
	It follows that $(x,y)$ also has finite projective dimension. Since $R$ is local and one-dimensional, the Auslander–Buchsbaum formula gives
	\[
	\pd_R((x,y)) + \depth_R((x,y)) = \depth(R)=1.
	\]
	But $(x,y)\subseteq R$ is torsion-free over the domain $R$, so $\depth_R((x,y))\ge 1$. Hence $\pd_R((x,y))=0$, and therefore $(x,y)$ is a free $R$-module of rank one. Thus $(x,y)$ is principal.
	Because $y\notin (x)$ by $(\ast)$, the generator must be $y$, so $(x,y)=(y)$. Hence $x=ry$ for some $r\in R$. Since $y\notin (x)$, the element $r$ is not a unit.
	From $(\ast)$ we also have $y^p = s x^p$ for some $s\in R$. Substituting $x=ry$ gives
	\[
	y^p = s(r y)^p = s r^p y^p.
	\]
	Since $R$ is a domain and $y\neq 0$, we can cancel $y^p$ to obtain
	\(
	1 = s r^p.
	\)
	Thus $r$ is a unit, a contradiction.
	Therefore $\pd_R(N)=\infty$, completing the proof.
\end{proof}

\begin{example}
	There exist $R$-modules $M,N$ of infinite projective dimension over the Artinian ring
	\(
	R=\overline{\mathbb{F}}_p[[X]]/(X^p)
	\)
	such that
	\(
	\pd_R(M\otimes_R N)=0.
	\)
\end{example}

\begin{proof}
	Let
	\[
	N:=R/xR=R/\mathfrak m, \qquad M:={}^{F^n}\!R
	\]
	for some $n\ge 1$.
	Then
	\(
	M\otimes_R N \cong R/(x)^{[p^n]} = R/(x^{p^n}).
	\)
	Since $x^p=0$ in $R$, we have $x^{p^n}=0$, and hence
	\(
	M\otimes_R N \cong R.
	\)
	Thus
	\(
	\pd_R(M\otimes_R N)=0.
	\)
	Because $R$ is not regular, Kunz’s theorem implies
	\(
	\pd_R(M)=\infty,
	\)
	and clearly $N=R/\mathfrak m$ also has infinite projective dimension since $R$ is not a field.
\end{proof}
\medskip
\section{Projective dimension and chain condition}

The following improves \cite[Proposition II.3]{S}:

\begin{observation}
	Let $(R,\mathfrak{m})$ be a UFD. Then any chain of prime ideals extends to length $3$.  
	In particular, a $3$-dimensional UFD satisfies the strict chain condition (i.e., every saturated chain of prime ideals has length $3$).
\end{observation}

\begin{proof}
	We prove the particular case. Let $0 \subsetneq P \subsetneq \mathfrak{m}$ be a strict chain of prime ideals. Since $0 \subsetneq P$ is strict, $P$ has height one. By the UFD property, $P = (x)$ for some $x \in R$. Let $y \in \mathfrak{m} \setminus P$. Let $Q \in \Spec(R)$ be minimal over $(x,y)$. Since $(x) \subset (x,y) \subseteq \mathfrak{m}$ is strict, we must have $Q = \mathfrak{m}$. Thus $\sqrt{(x,y)} = \mathfrak{m}$, so $x,y$ form part of a system of parameters. Hence $\dim (R/(x,y)) = 1$, while $\dim(R)=3$, so $\operatorname{ht}(x,y) = 2$. This shows that any saturated chain in $\Spec(R)$ has length $3$.
\end{proof}

The following is another example.

\begin{proposition}\label{s}
	Let $R$ be a non-catenary domain of dimension $3$. Then any nonzero prime ideal occurring in a bad saturated chain has infinite projective dimension.
\end{proposition}

\begin{proof}
	Let $0 \subsetneq P \subsetneq \mathfrak{m}$ be a strict chain of prime ideals. If $\pd_R(\mathfrak{m})$ were finite, then $\pd_R(k)<\infty$, and hence $R$ would be regular. But regular rings are catenary, a contradiction. Thus $\pd_R(\mathfrak{m})=\infty$.
	
	Suppose, toward a contradiction, that $\pd_R(P)<\infty$. Since $0 \subsetneq P$ is strict, $P$ has height one. Also, $\depth_R(R/P)>0$ because $P$ is prime. Since $\depth(R)=3$, the Auslander–Buchsbaum formula gives
	\[
	\pd(R/P)=\depth(R)-\depth(R/P)\leq 3-1=2.
	\]
	By \cite[Corollary 2]{ABU}, it follows that $P=(x)$ for some $x \in R$. Let $y \in \mathfrak{m} \setminus P$, and let $Q \in \Spec(R)$ be minimal over $(x,y)$. As before, $Q=\mathfrak{m}$, so $\sqrt{(x,y)}=\mathfrak{m}$. By the generalized principal ideal theorem,
	\( \dim(R)=\operatorname{ht}(\mathfrak{m})\leq 2,
	\)
	a contradiction. Hence $\pd_R(P)=\infty$.
\end{proof}

\begin{corollary}
	Let $R$ be a domain of dimension $d$ with the following strict saturated chain of primes:
	\[
	0=\mathfrak{p}_0 \subsetneq \cdots \subsetneq \mathfrak{p}_{d-2}\subsetneq \mathfrak{m}.
	\]
	Then $\pd(\mathfrak{p}_{d-2})$ is infinite.
\end{corollary}

\begin{proof}
	By Proposition \ref{s}, we may assume $d\geq 4$. Suppose, toward a contradiction, that $\pd(\mathfrak{p}_{d-2})$ is finite. Note that
	\(
	d-2\leq \operatorname{ht}(\mathfrak{p}_{d-2})\leq d-1.
	\) First assume $\operatorname{ht}(\mathfrak{p}_{d-2})=d-1$. Since $\pd_{R_{\mathfrak{p}_{d-2}}}(\mathfrak{p}_{d-2}R_{\mathfrak{p}_{d-2}})<\infty$, the ring $R_{\mathfrak{p}_{d-2}}$ is regular, hence catenary. But
	\[
	0=\mathfrak{p}_0R_{\mathfrak{p}_{d-2}} \subsetneq \cdots \subsetneq \mathfrak{p}_{d-2}R_{\mathfrak{p}_{d-2}}
	\]
	is a strict saturated chain in $\Spec(R_{\mathfrak{p}_{d-2}})$ of length less than $\dim(R_{\mathfrak{p}_{d-2}})$, a contradiction. Hence $\operatorname{ht}(\mathfrak{p}_{d-2})\neq d-1$, so $\operatorname{ht}(\mathfrak{p}_{d-2})= d-2$.
	The ring $R$ is not Cohen–Macaulay; otherwise it would be catenary. Thus Proposition \ref{10} applies, and $R/\mathfrak{p}_{d-2}$ is perfect. By the Peskine–Szpiro theorem, the grade conjecture holds for perfect modules, so
	\[
	\grade(R/\mathfrak{p}_{d-2})+\dim(R/\mathfrak{p}_{d-2})=\dim R=d.
	\]
	Since there is no prime strictly between $\mathfrak{p}_{d-2}$ and $\mathfrak{m}$, we have $\dim(R/\mathfrak{p}_{d-2})=1$, hence $\grade(R/\mathfrak{p}_{d-2})=d-1$. By Fact \ref{beder},
	\[
	\operatorname{ht}(\mathfrak{p}_{d-2})=\grade(R/\mathfrak{p}_{d-2})=d-1,
	\]
	a contradiction. Therefore $\pd_R(\mathfrak{p}_{d-2})=\infty$.
\end{proof}

\begin{remark}
	We construct a ring $R$ such that for every $\mathfrak{p} \in\Spec(R)\setminus\{ 0, \mathfrak{m}\}$, one has $\pd(R/\mathfrak{p}) = \infty$ but $\pd(R_{\mathfrak{p}}) < \infty$.
\end{remark}

\begin{proof}
	Let $A := k[X,Y,Z,W,T]/(X,Y)(Z,W)$. By work of Heitmann \cite{hei}, there exists a UFD $R$ with isolated singularities such that $\widehat{R} = A$.
	Let $B := k[x,y,z,w,t]$. A minimal free resolution of $(x,y)(z,w)$ over $B$ is
	\[
	0 \to B \to B^4 \to B^4 \to (x,y)(z,w) \to 0.
	\]
	By the Auslander–Buchsbaum formula,
	\[
	\depth(A) = \depth(B) - \pd_B(A) = 5 - 3 = 2.
	\]
	By \cite[Ex.~1.2.26]{BH}, $\depth(A) = \depth(R) = 2$.
	Let $\mathfrak{p} \in \Spec(R)$ and suppose $\pd(R/\mathfrak{p}) < \infty$. Then $\operatorname{ht}(\mathfrak{p}) \in \{2,3\}$.
First, we deal with $\operatorname{ht}(\mathfrak{p}) = 2$. Thus,
	\[
	2 = \operatorname{ht}(\mathfrak{p}) \stackrel{\ref{beder}}= \grade(R/\mathfrak{p})
	\leq \pd(R/\mathfrak{p}) = \depth(R) - \depth(R/\mathfrak{p}) \leq 2 - 1 = 1,
	\]
	a contradiction.
It remains to consider the case $\operatorname{ht}(\mathfrak{p}) = 3$. The same argument gives us
	\[
	3 = \operatorname{ht}(\mathfrak{p}) \stackrel{\ref{beder}}= \grade(R/\mathfrak{p})
	\leq \pd(R/\mathfrak{p}) \leq 1,
	\]
	again a contradiction. Hence $\pd(R/\mathfrak{p}) = \infty$ for all nonzero nonmaximal primes $\mathfrak{p}$.
\end{proof}

\end{document}